\documentclass[a4paper,12pt]{article}

\usepackage[latin1]{inputenc}
\usepackage[english]{babel}
\usepackage{amssymb}
\usepackage{amsmath}
\usepackage{latexsym}
\usepackage{amsthm}
\usepackage[pdftex]{graphicx}
\usepackage{tikz} %questo package va messo dopo graphicx
\usepackage{hyperref}
\usepackage{pgfplots}  
\usepackage{mathtools}
\pgfplotsset{width=6.6cm,compat=1.7}

\usepackage{array}
\usepackage{booktabs}

\usepackage{placeins}

\usetikzlibrary{cd}

\usetikzlibrary{decorations.pathreplacing}

\DeclareMathOperator*{\wep}{wp}

 \theoremstyle{plain}
 \newtheorem{thm}{Theorem}[section]
 \newtheorem{cor}[thm]{Corollary}
 \newtheorem{lem}[thm]{Lemma}
 \newtheorem{prop}[thm]{Proposition}
 \newtheorem{question}[thm]{Question}

 \theoremstyle{definition}
 \newtheorem{defn}[thm]{Definition}
 \newtheorem{example}[thm]{Example}

 \theoremstyle{remark}
 \newtheorem{oss}[thm]{Remark}
\title{Involutions avoiding $4321$ and another pattern of length four}
\date{}
\author{Marilena Barnabei $^\dagger$ \\
P.A.M. \\
Universit\`a di Bologna, 40126, ITALY \\
\texttt{marilena.barnabei@unibo.it}\\
https://orcid.org/0000-0001-9682-3834\and
Niccol\`o Castronuovo $^\dagger$ \thanks{corresponding author} \ \\
Universit\`a di Bologna, Campus di Cesena, 47521, ITALY \\
\texttt{niccolo.castronuovo2@unibo.it}\\
https://orcid.org/0000-0002-1054-0161\and
Matteo Silimbani $^\dagger$  \\
Istituto Comprensivo ``E. Rosetti'', Forlimpopoli, 47034, ITALY \\
\texttt{matteosilimbani@icrosetti.istruzioneer.it}\\
https://orcid.org/0000-0002-6617-2433
}

\begin{document}
\maketitle
\begin{abstract}
We enumerate all families of involutions avoiding a classical pattern of length four together with 4321. 
While the enumeration of permutations avoiding two  classical patterns is by now well understood, the case of pattern-avoiding involutions presents additional structural constraints that require dedicated techniques. In particular our main enumerative tool is Biane's bijection between involutions and labelled Motzkin paths.

The paper provides exact bivariate  generating functions (taking into account length and number of descents), and structural descriptions for all except one family. In the remaining case we provide a functional equation for the generating function and express it as an explicit continued fraction.

We make extensive use of experimental and software-supported methods. 
\end{abstract}



\vspace{0.5 cm}

\section{Introduction}\label{intro}

The study of pattern-avoiding permutations has developed into a central area of enumerative combinatorics over the last three decades.  
Given a permutation $\pi\in S_n$ and a pattern $\tau\in S_k$, we say that $\pi$ \emph{avoids} $\tau$ if no subsequence of $\pi$ is order-isomorphic to $\tau$.  
The systematic study of such classes began with the foundational work of Knuth~\cite{Kn2}, and later expanded through the seminal contributions of Simion and Schmidt~\cite{Si}, among many others.  
Comprehensive accounts can be found, for instance, in the books by B\'ona~\cite{Bo} and Kitaev~\cite{Ki}.  

A substantial line of research has focused on permutations avoiding \emph{two} patterns of a fixed length.  
For two patterns of length four, the classification of the Wilf-equivalence classes and the enumeration of the corresponding permutation families is almost complete ~\cite{AlbertAtkinsonBrignall2011,AlbertAtkinsonBrignall2012,AlbertAtkinsonVatter2009,AlbertAtkinsonVatter2014,AlbertHombergerPantoneSharVatter2018,Atkinson1998,AtkinsonSaganVatter2012,Bevan2016b,Bevan2016a,BloomVatter2016,Bona1998,Callan2013a,Callan2013b,Kremer2000,Kremer2003,KremerShiu2003,Le2005,Miner2016,MinerPantone2018,Pantone2017,Vatter2006,Vatter2012} with a few notable exceptions: for example, the class of permutations avoiding $4321$ and $4231$ has not been enumerated, and it is conjectured ~\cite{AlbertHombergerPantoneSharVatter2018} that its generating function does not satisfy any algebraic differential equation.  

In parallel with the general theory of pattern-avoiding permutations, several authors have investigated \emph{involutions} avoiding specific patterns.  
An involution is a permutation $\pi$ satisfying $\pi=\pi^{-1}$, and its fixed-point and transposition structure imposes additional restrictions on the appearance of patterns.  
Classical contributions include B\'ona~\cite{Bonainv},Guibert~\cite{Gu,Guibert2001},  Barnabei et al.~\cite{Ba,Ba6}, and, more recently, Bean et al.~\cite{BeanGuttmannPantone2026}, among others.  
The avoidance of patterns by involutions often exhibits connections with lattice paths, continued fractions, and orthogonal polynomials, making enumeration strongly dependent on bijective or analytic techniques.

Rather than simply counting the elements of a pattern class, it is
often natural to refine the enumeration by taking into account further
permutation statistics, such as the number of descents, inversions or
fixed points, or the major index; a considerable literature is devoted
to the joint distribution of such statistics over pattern-avoiding
permutations and involutions (see, e.g.,~\cite{Ba4, Ba5,jointdis}). In this spirit,
all our generating functions keep track not only of the length but
also of the number of descents.

Experimental methods have become a standard tool in enumerative
combinatorics: computer-generated data are routinely used to guess
structural characterizations, generating functions and recurrences,
which are then proved by combinatorial or analytic arguments (see, for
instance, the recent papers~\cite{Chern05012026,Lewis06012026,Mansour03082026}).

\medskip

\noindent
\textbf{This paper.}  
In this work we complete the enumeration of involutions avoiding two patterns of length four, one of which is 4321.   
For each pattern $\tau,$ $|\tau|=4$, we determine the generating function of the corresponding class of involutions and describe its structure. Most of the generating functions we obtain are rational, a few are algebraic, and the remaining one, which admits no closed form, is conjectured to be non-D-finite.

Our main enumerative tool is \emph{Biane's bijection} between involutions and \emph{labelled Motzkin paths} (see~\cite{BIANE} and~\cite{Ba6}), which gives a precise correspondence between pattern constraints and restrictions on the labels of the associated paths.  
This bijective framework allows us to derive exact generating functions and obtain structural insight into the relevant classes.

The analogous problem for involutions avoiding $3412$ together with
another pattern was addressed by Egge~\cite{EGGE2004451}, who developed a
general procedure, based on continued fractions and Chebyshev
polynomials, which handles all the additional patterns in a uniform way.
In the case of $4321$ no such uniform approach seems to be available:
the way the avoidance of a second pattern is reflected by Biane's
bijection strongly depends on the pattern itself, and each case requires
an ad hoc analysis of the forbidden configurations in the associated
Motzkin paths. On the other hand, this approach has the advantage of
keeping track of the descent statistic very naturally, since descents of
a $4321$-avoiding involution correspond to weak peaks, a purely local
feature of the associated path. As a consequence, for each family we
obtain bivariate generating functions refining the enumeration by length
according to the number of descents. We obtain rational or algebraic generating functions in all cases except $\tau=4231.$ In this case, the corresponding generating function satisfies a functional
equation that can be
expressed as an explicit continued fraction, and the experimental
evidence suggests that it is not D-finite.
% The appearance of a
% continued fraction should not come as a surprise: continued fractions
% arise naturally in enumerative combinatorics, especially in the
% enumeration of lattice paths, as shown in the seminal work of
% Flajolet~\cite{flaj}.

Notice that we considered only the cases in which the set $T$ of avoided patterns consists of only one pattern $\tau$ with, obviously, its inverse. However, our procedure perfectly works for every set of patterns $T$ 
and all the cases where $T$ consists of genuinely different patterns can be easily deduced from our results. 

Our approach has been experimental in the following sense. We used a dedicated Python script has been used systematically
to guess the structure of the set of Motzkin paths corresponding to each family. The computer algebra system \textsc{Mathematica}~\cite{Mathematica} has been used to perform algebraic manipulations, verify identities and check the results against direct generation of involutions. All the code used is made freely available in electronic form in a dedicated github page. 

Notice that none of the bivariate refinements, and most of the sequences, appear in the OEIS~\cite{Sl}.

\medskip

\noindent
\textbf{Organization of the paper.}
Section~\ref{prelim} recalls basic definitions and symmetries, and
Section~\ref{sec:4321-bijection} Biane's bijection.
Section~\ref{sec:methodology} describes our experimental methodology.
Section~\ref{4321} deals with $I_n(4321)$; Section~\ref{sec:classes_char}
with the classes whose connected paths are characterized by forbidden
subsequences; Section~\ref{beyond} with those requiring additional
conditions; Section~\ref{exceptional} with the two exceptional classes,
including $I_n(4321,4231)$. Section~\ref{sec:conclusion} contains
concluding remarks and open problems.

\section{Preliminaries and Notation}\label{prelim}

Let $S_n$ denote the symmetric group on $\{1,2,\dots,n\}$.  
A permutation $\pi \in S_n$ \emph{contains} a pattern $p \in S_k$ if some 
subsequence of $\pi$ of length $k$ has the same relative order as $p$; 
otherwise $\pi$ \emph{avoids} $p$.

Given a permutation $\pi\in S_n,$ an index $i$ with $1\leq i\leq n-1$ is called a \textit{descent} for $\pi$ if $\pi_i>\pi_{i+1}.$ The number of descents of a permutation $\pi$ is denoted by $\operatorname{des}(\pi).$

A non-empty permutation $\pi\in S_n$ is \emph{connected} (or
\emph{indecomposable}) if there is no $k<n$ such that
$\pi(\{1,\dots,k\})=\{1,\dots,k\}$.

A permutation $\pi \in S_n$ is an \emph{involution} if $\pi^2 = \mathrm{id}$, 
equivalently $\pi = \pi^{-1}$.  
We denote by
\[
I_n = \{\pi \in S_n : \pi = \pi^{-1}\}
\]
the set of involutions of length $n$.
 
For patterns $p_1,\dots,p_r$ we use
$I_n(p_1,\dots,p_r)$ to denote the set of involutions in $I_n$ avoiding all listed patterns.

\begin{oss}\label{conn_patt}
If a permutation $\pi$ contains the pattern $\tau$ and $\tau$ is connected, then the subsequence corresponding to $\tau$ must lie inside one of the connected components of $\pi.$ In other terms, a permutation $\pi$ avoids a connected pattern $\tau$ if and only if all the connected components of $\pi$ avoid $\tau.$  
\end{oss}

\subsection{The inverse and reverse--complement operations}

\paragraph{Inverse.}
For a permutation $\pi \in S_n$, the inverse permutation $\pi^{-1}$ is defined by
$\pi^{-1}(\pi(i)) = i$.  
The following observation is fundamental.

\begin{lem}
\label{lem:inverse-avoidance}
If $\pi \in I_n$ and $\tau$ is any pattern, then
\[
\pi \text{ avoids } \tau 
\quad\Longleftrightarrow\quad
\pi \text{ avoids } \tau^{-1}.
\]
\end{lem}

% \begin{proof}
% A subsequence of $\pi$ has pattern $p$ if and only if the corresponding 
% subsequence of $\pi^{-1} = \pi$ has pattern $p^{-1}$.  
% Thus $\pi$ avoids $p$ exactly when it avoids $p^{-1}$.
% \end{proof}

\begin{cor}
\label{cor:add-inverses}
For any patterns $\tau_1,\dots,\tau_r$,
\[
I_n(\tau_1,\dots,\tau_r) = I_n(\tau_1,\dots,\tau_r,\tau_1^{-1},\dots,\tau_r^{-1}).
\]
\end{cor}

\paragraph{Reverse--complement.}
Given $\pi \in S_n$, its \emph{reverse--complement} is the permutation
$\pi^{rc} \in S_n$ defined by
\[
\pi^{rc}(i) = n+1 - \pi(n+1-i).
\]
This operation reverses the one-line notation of $\pi$ and replaces each
value $x$ by $n+1-x$.  

\begin{lem}
\label{lem:rc-avoidance}
For any permutation $\pi \in S_n$ and any pattern $\tau$, one has
\[
\pi \text{ avoids } \tau 
\quad\Longleftrightarrow\quad
\pi^{rc} \text{ avoids } \tau^{rc}.
\]
\end{lem}

% \begin{proof}
% Geometrically, reverse--complement corresponds to a $180^\circ$ rotation
% of the permutation diagram.  
% A set of points has pattern $p$ if and only if its rotated copy has pattern
% $p^{rc}$; thus containment (and therefore avoidance) is preserved.
% \end{proof}

Since the map $\pi \mapsto \pi^{rc}$ preserves involutions and number of descents, we immediately obtain:

\begin{cor}
\label{cor:rc-bijection}
For any patterns $\tau_1,\dots,\tau_r$, the map
\[
\Phi : I_n(\tau_1,\dots,\tau_r) \longrightarrow I_n(\tau_1^{rc},\dots,\tau_r^{rc}), 
\qquad \Phi(\pi) = \pi^{rc},
\]
is a $\operatorname{des}$-preserving bijection.  In particular,
\[
|I_n(\tau_1,\dots,\tau_r)| = |I_n(\tau_1^{rc},\dots,\tau_r^{rc})|
\]
and $\operatorname{des}(\Phi(\pi))=\operatorname{des}(\pi).$
\end{cor}

\paragraph{Combined symmetries.}
The four operations
\[
\mathrm{id},\qquad (\cdot)^{-1},\qquad (\cdot)^{rc},\qquad 
(\cdot)^{-1} \circ (\cdot)^{rc}
\]
form a group of symmetries acting on permutations and patterns.
They preserve the property of being an involution and preserve pattern-avoidance
in the sense of Lemmas~\ref{lem:inverse-avoidance} and 
\ref{lem:rc-avoidance}.  

Thus two sets of patterns $(\tau_1,\dots,\tau_r)$ and $(\tau_1',\dots,\tau_r')$ are called 
\emph{equivalent} if one can be obtained from the other by applying the same 
symmetry to all patterns simultaneously.  
By Corollaries~\ref{cor:add-inverses} and~\ref{cor:rc-bijection}, all sets
$I_n(\tau_1,\dots,\tau_r)$ corresponding to patterns within the same symmetry class are in 
canonical bijection.

Our enumerative problem is thus reduced to the classes listed in Table \ref{tab:summary}.

\vspace{0.5 cm}

\section{Biane's bijection and $4321$-avoiding involutions}
\label{sec:4321-bijection}

One of the most effective tools for studying involutions avoiding the pattern $4321$ is a bijection due to Biane~\cite{BIANE}, which encodes involutions as labelled Motzkin paths.  
We briefly recall the construction in the form that is most convenient for our purposes, following the approach of Barnabei, Bonetti, and Silimbani~\cite{Ba6}.

\subsection{Motzkin paths}

A \emph{Motzkin path} of length $n$ is a lattice path in $\mathbb{Z}^2$ starting at $(0,0)$, ending at $(n,0)$, and consisting of the following steps:
\[
U = (1,1), \quad H = (1,0), \quad D = (1,-1),
\]
never going below the $x$-axis. We regard a Motzkin path $m$ of length $n$ as a word $m=s_1s_2\cdots s_n$
over the alphabet $\{U,H,D\}$.

A \emph{labelled Motzkin path} is a Motzkin path in which each down step is assigned a positive integer label not exceeding the height of its starting point.

We will repeatedly use the \emph{first return decomposition}: every
non-empty Motzkin path $m$ can be uniquely written either as $m=Hm'$ or
as $m=Um_1Dm_2$, where $m',m_1,m_2$ are (possibly empty) Motzkin paths
and $D$ is the first step of $m$ returning to the $x$-axis. This yields
the classical functional equation for the generating function
$M(x)=\sum_{m\in\mathcal M}x^{|m|}$ of Motzkin paths:
\begin{equation}\label{eq:first-return}
M(x)=1+xM(x)+x^2M(x)^2 .
\end{equation}

A \emph{connected} Motzkin path is a non-empty Motzkin path
whose only return to the $x$-axis is its final point.

\begin{defn}\label{def:subsequence}
Let $m=s_1s_2\cdots s_n$ be a Motzkin path and let $w=w_1w_2\cdots w_k$
be a word over $\{U,H,D\}$.
\begin{itemize}
    \item An \emph{occurrence} of $w$ in $m$ is a subsequence
    $s_{i_1}s_{i_2}\cdots s_{i_k}$ of $m$, with $i_1<i_2<\cdots<i_k$,
    such that $s_{i_j}=w_j$ for every $j$. We say that $m$
    \emph{contains} $w$ if $m$ has at least one occurrence of $w$, and
    that $m$ \emph{avoids} $w$ otherwise.
    \item An occurrence of $w$ in $m$ is a \emph{factor} of $m$ if its
    steps are consecutive, namely $i_{j+1}=i_j+1$ for every $j$.
\end{itemize}
\end{defn}

\begin{oss}
We stress that, when we say that a path contains or avoids a word $w$,
the steps of an occurrence of $w$ need not be consecutive. For instance,
the path $UHUDHD$ contains $HD$, $UUDD$ and $HHD$, while it avoids
$DU$. Whenever consecutive steps are required, we explicitly speak of
factors. When needed, we mark the steps of a particular occurrence
with accents, as in $\hat U\hat H\hat D\overline D$.
\end{oss}

\subsection{Definition of the bijection}

\begin{defn}[Biane's bijection]
\label{def:biane}
Given $\pi \in I_n$, scan its one-line notation from left to right and associate a step to each position $i$:
\begin{itemize}
    \item If $i$ is a fixed point (i.e.\ $\pi(i)=i$), we associate a horizontal step $H$.
    \item If $i < \pi(i)$ (opening of a transposition), we associate an up step $U$.
    \item If $i > \pi(i)$ (closing of a transposition), we associate a down step $D$, labelled by the rank of $\pi(i)$ among the currently open transpositions.
\end{itemize}
\end{defn}

During the scanning phase, a transposition is said to be \textit{active} if the up step corresponding to its smaller element has already been read, but the down step corresponding to its larger element has not. 

\begin{oss}
The height of the path at a given step corresponds to the number of currently open transpositions.  
The labels on down steps encode how these transpositions are matched.
\end{oss}

\begin{thm}[\cite{BIANE}]
\label{thm:biane}
The above construction defines a bijection between involutions of size $n$ and labelled Motzkin paths of length $n$.
\end{thm}

\medskip

For involutions avoiding $4321$, the bijection simplifies considerably:

\begin{prop}
\label{prop:unitary-labels}
An involution avoids $4321$ if and only if all labels in the associated Motzkin path are equal to $1$.
\end{prop}

\noindent
Thus $4321$-avoiding involutions are in bijection with \emph{unlabelled} Motzkin paths. We will denote the unlabelled Motzkin path corresponding to $\pi\in I_n(4321)$ by $\beta(\pi).$

\begin{prop}\label{prop:increasing-runs}
Let $\pi\in I_n(4321)$. Then the entries of $\pi$ corresponding to the
$U$ steps (respectively, $D$ steps, $H$ steps) of $\beta(\pi)$ form an
increasing subsequence of $\pi$.
\end{prop}

\begin{proof}
If $i<j$ are two positions corresponding to $U$ steps with
$\pi(i)>\pi(j)$, then $i<j<\pi(j)<\pi(i)$ and the entries in positions
$i,j,\pi(j),\pi(i)$ are $\pi(i),\pi(j),j,i$, an occurrence of $4321$.
The entries corresponding to $D$ steps are the positions of the $U$
steps, listed in the order in which the transpositions are closed, which
is increasing by the previous argument. Finally, the entries corresponding
to $H$ steps are fixed points, hence trivially increasing.
\end{proof}

\begin{oss}\label{oss:decreasing}
As a consequence, if $\pi\in I_n(4321)$, the entries of any decreasing
subsequence of $\pi$ correspond to steps of $\beta(\pi)$ of pairwise
different kinds.
\end{oss}

\begin{defn}
A \emph{weak peak} of a Motzkin path $m$ is an occurrence of two
consecutive steps of the form $UD$, $UH$ or $HD$. We denote by
$\wep(m)$ the number of weak peaks of $m$. A \textit{weak valley} is an occurrence in $m$ of two consecutive steps of the form $DU,$ $DH$ or $HU.$
\end{defn}

The following lemma is a direct consequence of the definition of the Biane map, when restricted to 4321-avoiding involutions. 

\begin{lem}\label{discese_picchi_deboli}
For every $\pi\in I_n(4321)$ we have $\operatorname{des}(\pi)=\wep(\beta(\pi))$.
More precisely, $i$ is a descent of $\pi$ if and only if the $i$-th and
$(i+1)$-st steps of $\beta(\pi)$ form a weak peak.
\end{lem}

For $k<n$, we have $\pi(\{1,\dots,k\})=\{1,\dots,k\}$ if and only if no transposition is active after position $k$, that is, if and only if $\beta(\pi)$ touches the $x$-axis after its $k$-th step.
 Hence $\pi$ is connected if and only if
$\beta(\pi)$ is.  Moreover, the
decomposition of $\pi$ into connected components corresponds to the
decomposition of $\beta(\pi)$ into its connected factors.

We point out also the following known fact. Given any involution $\pi\in I_n(4321),$ the Motzkin path $\beta(\pi^{rc})$ corresponding to $\pi^{rc}$ is obtained from $\beta(\pi)$ by reflection in a vertical line.

\subsection{An example}

We illustrate the bijection with a simple example.

\begin{example}
Consider the involution
\[
\pi = (1\,3)(2\,5)(4)(6) = 351426 \in I_6(4321).
\]
Scanning $\pi$ from left to right, positions $1$ and $2$ open the
transpositions $(1\,3)$ and $(2\,5)$, position $3$ closes $(1\,3)$,
position $4$ is a fixed point, position $5$ closes $(2\,5)$ and
position $6$ is a fixed point. Hence $\beta(\pi)=UUDHDH$, shown in
Figure~\ref{fig:example-bijection}. Since $\pi$ avoids $4321$, every
down step closes the earliest open transposition, so all labels are
equal to $1$ and can be omitted.
Notice that the entries $3,5$ (corresponding to $U$ steps), $1,2$
($D$ steps) and $4,6$ ($H$ steps) form increasing subsequences of $\pi$,
and that the two descents of $\pi$, in positions $2$ and $4$, correspond
to the two weak peaks $UD$ and $HD$ of $\beta(\pi)$.
\end{example}

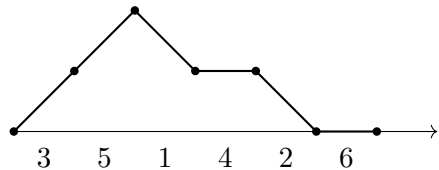
\begin{figure}[h]
\centering

\begin{tikzpicture}[scale=0.8]
% axes
\draw[->] (0,0) -- (7,0);

% path
\draw[thick]
(0,0) -- (1,1)   % U
-- (2,2)         % U
-- (3,1)         % D
-- (4,1)         % H
-- (5,0)         % D
-- (6,0);        % H

% points
\foreach \x/\y in {0/0,1/1,2/2,3/1,4/1,5/0,6/0}
    \fill (\x,\y) circle (2pt);

% entries of pi below the steps
\foreach \x/\v in {0.5/3,1.5/5,2.5/1,3.5/4,4.5/2,5.5/6}
    \node[below] at (\x,-0.1) {\small $\v$};

\end{tikzpicture}

\caption{The involution $\pi=(1\,3)(2\,5)(4)(6)=351426$ and its
associated Motzkin path $\beta(\pi)=UUDHDH$.}
\label{fig:example-bijection}
\end{figure}

\section{Experimental methodology}
\label{sec:methodology}

In this section we explain the experimental methodology that we followed for every set of patterns to be avoided together with 4321. Notice that we report only the cases in which $T$ consists of only one pattern $\tau$ with, possibly, its inverse. However, the procedure applies to any set $T.$

For every set of patterns $T,$
the characterizations of the Motzkin paths associated with the classes
$I_n(4321,T)$  were first
obtained experimentally, and then proved. In this section we describe
the computational procedure we followed, which is the same for all the
sets $T$ considered.

In many cases, the structure of a non-connected involution
in $I_n(4321,T)$ is easily described in terms of its connected
components (see Remark~\ref{conn_patt} and, e.g., Lemma~\ref{Lemma_1243} and
Theorem~\ref{1324_notc}). Therefore, the experimental part of our work
focuses on connected involutions, where the combinatorics of the
problem is concentrated.

\subsection{Detecting forbidden subsequences}

For each set $T$ we
proceeded as follows, by means of a Python script.

\begin{enumerate}
    \item We generated all connected involutions in $I_n(4321,T)$ for
    every $n$ up to a given fixed value $N$, say $N=20.$
    \item For each of them we computed the associated Motzkin path
    $\beta(\pi)$, regarded as a word over $\{U,H,D\}$. By
    Proposition~\ref{prop:unitary-labels}, labels can be omitted.
    \item We determined the set $\mathcal W^N_L(T)$ of the words $w$
    over $\{U,H,D\}$ of length at most $L$ that do not occur as a
    subsequence (in the sense of Definition~\ref{def:subsequence}, hence
    not necessarily consecutively) of any of the paths obtained in the
    previous step, and that are minimal with respect to this property,
    namely that contain no shorter word with the same property.
\end{enumerate}

We will call this procedure \texttt{guess[T]}.

\subsection{From conjectures to theorems}

The output of the procedure above is a conjecture of the form
\begin{quote}
\emph{a connected involution $\pi\in I_n(4321)$ avoids every pattern in $T$ if and
only if $\beta(\pi)$ avoids every word in $\mathcal W^N_L(T)$, for a given $L$ and a sufficiently large $N.$}
\end{quote}
The data only support this statement for $n\le N$ and for short words; its proof for every
$n$ requires a combinatorial argument that exploits the properties of
Biane's bijection,
and that has to be tailored to the specific patterns in $T$. 

In many cases the forbidden subsequences directly yield an explicit
description of the admissible connected paths as a finite list of
shapes, such as $U^aH^bD^a$ (see, e.g.,
Theorem~\ref{connected_1243}). In a few cases the forbidden
subsequences alone do not characterize the class, and additional
conditions, for instance on the heights of some steps, are needed (the cases of $T=\{1324\}$ and $T=\{2341\}$). 
In these cases the
experimental data are used as a guide to identify the correct
conditions.

In the hardest cases ($T=\{3421,4312\}$ and $T=\{4231\}$), our computation does not provide any set $\mathcal W^N_L(T)$ since, in these cases, the paths cannot be described as avoiding particular subsequences. In these cases, ad hoc techniques have to be used.

\subsection{Generating functions}

Once the structure of the admissible paths has been established, we
derive the bivariate generating function 
\[
F_T(x,y)=\sum_{n\ge0}\ \sum_{\pi\in I_n(4321,T)} x^n y^{\operatorname{des}(\pi)}
\]
by counting the corresponding paths according to their length and
number of weak peaks, in view of Lemma~\ref{discese_picchi_deboli}.
Depending on the case, this is done either by directly summing over the
shapes of the admissible paths, or by means of a system of functional
equations obtained from the first return
decomposition~\eqref{eq:first-return}. Algebraic manipulations were
performed with the computer algebra system \textsc{Mathematica}~\cite{Mathematica}. As a
final check, for each set of patterns $T$ we compared the coefficients of
the Taylor expansion of $F_T(x,y)$ with the joint distribution of
length and descents obtained by exhaustive generation of the
involutions in $I_n(4321,T).$

\medskip
\noindent\textbf{Notation.}
To avoid repetitions, throughout the following subsections we adopt
the following conventions. In each subsection, $T$ denotes the set of
patterns avoided, together with $4321$, by the involutions studied
there.
We will always denote by $F_T(x,y)$
the generating function of the whole set of involutions
avoiding $4321$ and all the patterns in $T$, counted by length and
number of descents, and we denote by $G_T(x,y)$ the analogous
generating function of the \emph{connected} involutions in this set.
By Lemma~\ref{discese_picchi_deboli}, $F_T(x,y)$
and $G_T(x,y)$ are also the generating functions of the corresponding
Motzkin paths (respectively, connected Motzkin paths), counted by
length and number of weak peaks.
\medskip

\section{$I_n(4321)$}\label{4321}

We begin with the study of the distribution of descents over the set of all \(4321\)-avoiding involutions.
Recall that, by the Biane's bijection, this is equivalent to studying the distribution of weak peaks over the set of Motzkin paths. 

Let
$$
M(x,y)=\sum_{m\in\mathcal{M}}x^{|m|}y^{\wep(m)},
$$

where \(\mathcal{M}\) denotes the set of Motzkin paths, \(|m|\) is the length of \(m\), and \(\wep(m)\) is the number of weak peaks of \(m\).

We further define \(M_i(x,y)\), \(M_f(x,y)\), \(M_{if}(x,y)\), and \(M_0(x,y)\) to be the generating functions for Motzkin paths that, respectively,

\begin{itemize}
\item start with an \(H\)-step but do not end with an \(H\)-step;
\item end with an \(H\)-step but do not start with an \(H\)-step;
\item both start and end with an \(H\)-step;
\item neither start nor end with an \(H\)-step.
\end{itemize}

Note that every Motzkin path contributes exactly to one of these generating functions. In particular, the one-step path \(H\) is included in \(M_{if}(x,y)\), while the empty path is included only in \(M_0(x,y)\).

The following system of functional equations is easily deduced by applying the first-return decomposition to the Motzkin paths:

$$
\begin{aligned}
M_i(x,y) &= x\bigl(M_i(x,y)+M_0(x,y)-1\bigr),\\
M_f(x,y) &= x\bigl(M_f(x,y)+M_0(x,y)-1\bigr),\\
M_{if}(x,y) &= x^2\bigl(M_i(x,y)+M_0(x,y)+M_{if}(x,y)+M_f(x,y)\bigr)+x,\\
M_0(x,y) &= 1+x^2\Bigl(M_0(x,y)-1+y+yM_i(x,y)+yM_f(x,y)+y^2M_{if}(x,y)\Bigr)\cdot \\ &
\bigl(M_0(x,y)+M_i(x,y)\bigr).
\end{aligned}
$$

The previous system of equations can be solved by any computer algebra system to give, for
$M(x,y)=M_i(x,y)+M_f(x,y)+M_{if}(x,y)+M_0(x,y),$ the following expression
$$
M(x,y)=\frac{N(x,y)}{D(x,y)}
$$
where
$$
\begin{aligned}
N(x,y)={}&-x^3y^2+(2x^3-x^2)y-x^3+x^2-x+1-\sqrt{P(x,y)},
\end{aligned}
$$
$$
D(x,y)
=
2x^4 y^2-4(x^4-x^3)y+2x^4-4x^3+2x^2,
$$
with
$$
\begin{aligned}
P(x,y)=&x^6y^4-2(2x^6-x^5)y^3+(6x^6-6x^5-x^4-2x^3)y^2 \\
&-2(2x^6-3x^5-x^4+x^3+x^2)y \\ &+x^6-2x^5-x^4+4x^3-x^2-2x+1.
\end{aligned}
$$

\section{Classes characterized by forbidden subsequences}\label{sec:classes_char}

In this Section we study all the families of involutions avoiding 4321 and another pattern of length 4 in which the characterization of the corresponding connected Motzkin paths turns out to be exactly the one provided by procedure \texttt{guess[T]}.

\subsection{$I_n(4321,1234)$}\label{1234}

For $n\geq 10,$ there are no involutions of this type by the Erd\H{o}s--Szekeres theorem~\cite{Erd_Sz}.

\subsection{$I_n(4321,1243)$}\label{1243}

\begin{thm}\label{connected_1243}
A connected involution $\pi$ avoids 4321 and 1243 if and only if the 
connected
path $\beta(\pi)$ 
avoids $HHU,$ $DDU,$ $DHU$ and $DDH.$

Those paths are precisely the paths of the following forms
\begin{itemize}
    \item 
    \textbf{Type 0}
    $\qquad H$ or $UH^aD,$ with $a\geq 0,$
    \item 
    \textbf{Type 1} $\qquad U^aDU^bH^cD^{a+b-1}$ with $a>1, b>0, c\geq 0,$
    \item \textbf{Type 2} $\qquad U^aHU^bDU^cH^dD^{a+b+c-1} $ with $a>0,b\geq 0,c>0,d\geq 0,$
    \item
    \textbf{Type 3} $\qquad U^aH^bDH^cD^{a-1} $ with $a>1,b\geq 0, c\geq 0,$
    \item \textbf{Type 4} $\qquad U^aHU^bH^cDH^dD^{a+b-1}$ with $a>0,b>0,c\geq 0, d\geq 0.$
\end{itemize}
\end{thm}
\proof
First of all, observe that a path $m$ is of one of the previous types if and only if it avoids each of the patterns HHU, DDU, DHU and DDHD. Notice also that a connected path avoids DDHD if and only if it avoids DDH.

If the connected path $m$ contains one of the patterns HHU, DDU, DHU and DDHD, it is easily seen that the involution $\beta^{-1}(m)$ contains 1243. 

On the other hand,   each of the paths of these forms corresponds under $\beta^{-1}$ to a connected involution avoiding 1243.
\endproof

\begin{lem}\label{Lemma_1243}
 A non-empty involution $\pi\in I_n(4321)$ avoids $1243$ if and only if the path $\beta(\pi)$ is of one of the following  types
 \begin{itemize}
     \item \textbf{Type A} $\qquad mH^t,$ where $m$ is connected of type 0,1,2,3 or 4 and $t\geq 0,$ 
     \item \textbf{Type B}  $\qquad m'U^aH^bD^aH^c$ where $m'=H,UD$ or $UHD$ and $a>0,b\geq 0, c\geq 0.$
 \end{itemize}
   
\end{lem}
\proof
Suppose that $\beta(\pi)$ has connected components $c_1,c_2,\ldots,c_k.$ 
If $c_1$ corresponds to an involution with an increasing sequence of length at least two, all the other connected components of $\beta(\pi)$ must be $H.$

Otherwise, $c_1$ can only be $H,$ $UD$ or $UHD$ and the path given by $c_2c_3\ldots c_k$  corresponds to an involution avoiding 132. 
\endproof

\begin{thm}\label{GF_1243}
We have $$F_T(x,y)=1+\frac{G_T(x,y)}{1-x}+(x+x^2y+x^3y^2)\cdot \left(\frac{x^2y}{(1-x)(1-x^2)}+\frac{x^3y^2}{(1-x^2)(1-x)^2}\right),$$
where
\begin{align*}
G_T(x,y)=& x+x^2y+\frac{x^3y^2}{1-x}\\
&+ \frac{x^6y^2}{(1-x^2)^2}+\frac{x^7y^3}{(1-x^2)^2(1-x)}\\
&+
\frac{x^7y^3}{(1-x^2)^3}+\frac{x^8y^4}{(1-x^2)^3(1-x)}+\frac{x^7y^3}{(1-x^2)^2}+\frac{x^8y^4}{(1-x^2)^2(1-x)}
\\
&
+\frac{x^4y}{(1-x^2)}+2\frac{x^5y^2}{(1-x^2)(1-x)}+\frac{x^6y^3}{(1-x^2)(1-x)^2}
\\
&
+\frac{x^5y^2}{(1-x^2)^2}+2\frac{x^6y^3}{(1-x^2)^2(1-x)}+\frac{x^7y^4}{(1-x^2)^2(1-x)^2}.
\end{align*}
\end{thm}
\proof

In order to keep track of the number of descents of the involution associated with a path, we distinguish between the cases in which the path contains a nonempty sequence of horizontal steps in a given position and those in which such a sequence is absent. Indeed, the presence of a nonempty block of horizontal steps contributes one additional descent to the corresponding involution.

Taking this observation into account, it is immediate to verify that the function $G_T(x,y)$ enumerates the connected paths described in Theorem~\ref{connected_1243}. More precisely, the first line of the expression for $G_T(x,y)$ corresponds to paths of Type~0, the second line to paths of Type~1, and so on.

On the other hand, the generating function $F_T(x,y)$ enumerates all paths corresponding to involutions avoiding both $4321$ and $1243$, as characterized in Lemma~\ref{Lemma_1243}. In particular, paths of Type~A are accounted for by the second summand of $F_T(x,y)$, whereas paths of Type~B are accounted for by the third summand.

This completes the proof.

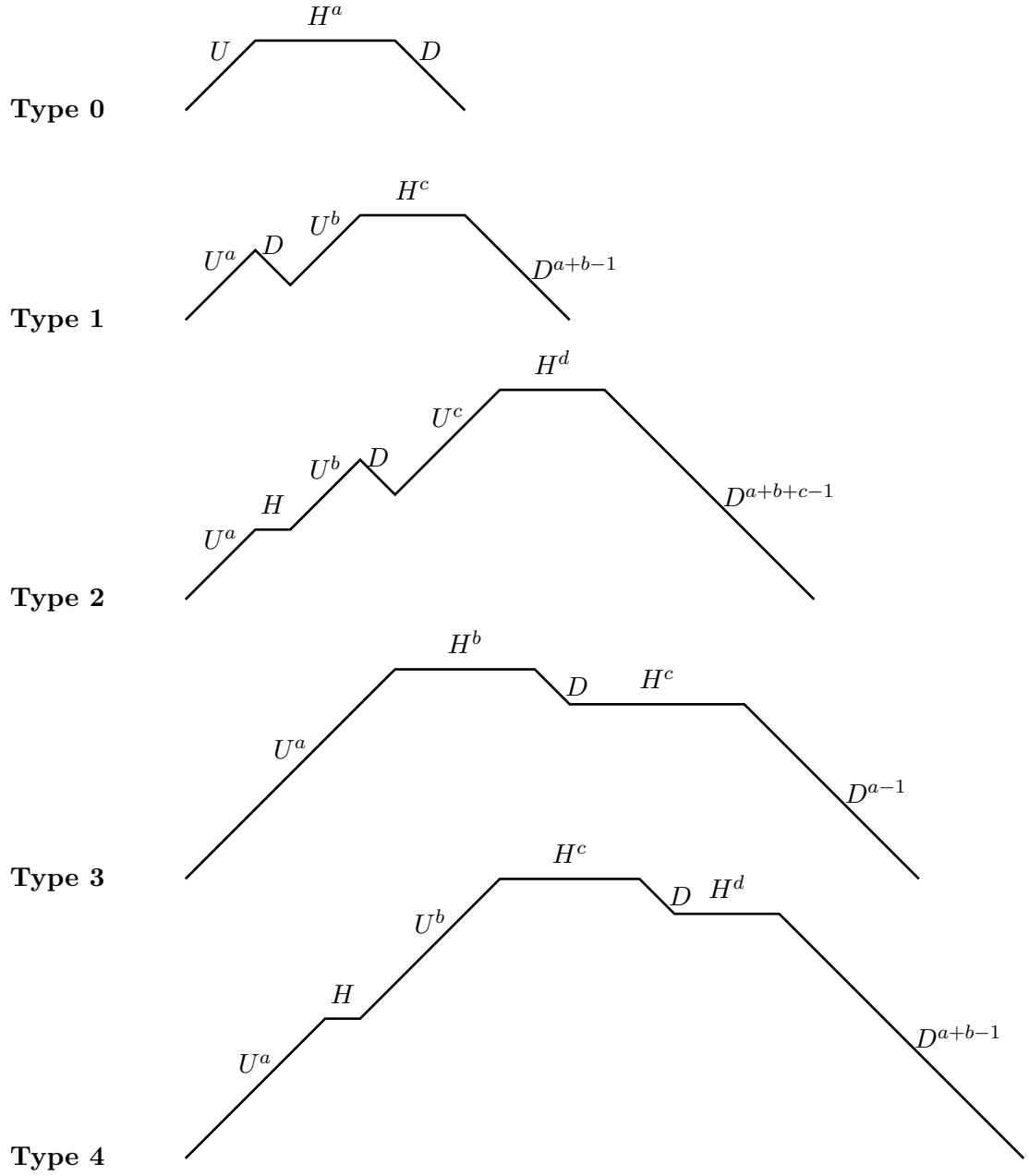
\endproof

%Q=(x^6y^2)(1+xy/(1-x))/((1-x^2)^2)+(x^7y^3)(1+xy/((1-x)))/((1-x^2)^3)+(x^7y^3)(1+xy/(1-x))/((1-x^2)^2)+(x^4y)(1+2xy/(1-x)+x^2y^2/(1-x)^2)/((1-x^2))+(x^5y^2)(1+2xy/(1-x)+x^2y^2/(1-x)^2)/((1-x^2)^2)+x+x^2y+(x^3y^2)/(1-x)
%p=12
%tot=1+Q/(1-x)+(x+x^2y+x^3y^2)((x^2y)/((1-x)(1-x^2))+(x^3y^2)/((1-x^2)(1-x)^2))
%tA=taylor(tot,(x,0),p)# funziona: Q sono le connesse che evitano 4321 e 1243, contaote per lung (x) e des (y) e tot il totale.

%\begin{figure}
\begin{tikzpicture}[
    line width=1pt,
    every node/.style={font=\small}
]

%------------------------------------------------
% Type 0
%------------------------------------------------
\node[left] at (-1,0) {\textbf{Type 0}};

\draw
(0,0)
 -- node[midway,above=2pt] {$U$} (1,1)
 -- node[midway,above=2pt] {$H^a$} (3,1)
 -- node[midway,above=2pt] {$D$} (4,0);

%------------------------------------------------
% Type 1
%------------------------------------------------
\node[left] at (-1,-3) {\textbf{Type 1}};

\draw
(0,-3)
 -- node[midway,above=2pt] {$U^a$} (1,-2)
 -- node[midway,above=1pt] {$D$} (1.5,-2.5)
 -- node[midway,above=2pt] {$U^b$} (2.5,-1.5)
 -- node[midway,above=2pt] {$H^c$} (4,-1.5)
 -- node[midway,right=1pt] {$D^{a+b-1}$} (5.5,-3);

%------------------------------------------------
% Type 2
%------------------------------------------------
\node[left] at (-1,-7) {\textbf{Type 2}};

\draw
(0,-7)
 -- node[midway,above=2pt] {$U^a$} (1,-6)
 -- node[midway,above=2pt] {$H$} (1.5,-6)
 -- node[midway,above=2pt] {$U^b$} (2.5,-5)
 -- node[midway, above] {$D$} (3,-5.5)
 -- node[midway,above=2pt] {$U^c$} (4.5,-4)
 -- node[midway,above=2pt] {$H^d$} (6,-4)
 -- node[midway,right=1pt] {$D^{a+b+c-1}$} (9,-7);

%------------------------------------------------
% Type 3
%------------------------------------------------
\node[left] at (-1,-11) {\textbf{Type 3}};

\draw
(0,-11)
 -- node[midway,above=2pt] {$U^a$} (3,-8)
 -- node[midway,above=2pt] {$H^b$} (5,-8)
 -- node[midway,right=1pt] {$D$} (5.5,-8.5)
 -- node[midway,above=2pt] {$H^c$} (8,-8.5)
 -- node[midway,right=1pt] {$D^{a-1}$} (10.5,-11);

%------------------------------------------------
% Type 4
%------------------------------------------------
\node[left] at (-1,-15) {\textbf{Type 4}};

\draw
(0,-15)
 -- node[midway,above=2pt] {$U^a$} (2,-13)
 -- node[midway,above=2pt] {$H$} (2.5,-13)
 -- node[midway,above=2pt] {$U^b$} (4.5,-11)
 -- node[midway,above=2pt] {$H^c$} (6.5,-11)
 -- node[midway,right=0.5 pt] {$D$} (7,-11.5)
 -- node[midway,above=2pt] {$H^d$} (8.5,-11.5)
 -- node[midway,right=1pt] {$D^{a+b-1}$} (12,-15);

\end{tikzpicture}
%\caption{Schematic representatives of the five Motzkin-path types. Horizontal segments labeled $H^k$ represent $k$ consecutive horizontal steps, while $U^k$ and $D^k$ indicate runs of up- and down-steps.}
Schematic representatives of the five connected Motzkin-path types corresponding to involutions avoiding 1243 by Theorem \ref{connected_1243}. Horizontal segments labeled $H^k$ represent $k$ consecutive horizontal steps, while $U^k$ and $D^k$ indicate runs of up- and down-steps.
%\end{figure}

% \begin{figure}[h]
% \centering

% \begin{tikzpicture}[scale=0.9]

% %------------------------------------------------
% % M0 decomposition
% %------------------------------------------------
% \node at (0,4.5) {$\mathcal M_0$};

% \draw[thick]
% (0,4) -- (1,5)
% node[midway,left] {$U$}
% -- (5,5)
% node[midway,above] {$m\in \mathcal M_0\cup\mathcal M_L$}
% -- (6,4)
% node[midway,right] {$D$}
% -- (8,4)
% node[midway,below] {$m'\in \mathcal M_0\cup\mathcal M_L$};

% \node at (3,-0.2) {};

% \node at (4,3.2)
% {$\displaystyle UDm'
% \quad\text{or}\quad
% UmDm'$};

% %------------------------------------------------
% % ML decomposition
% %------------------------------------------------
% \node at (0,1.8) {$\mathcal M_L$};

% \draw[thick]
% (0,1.3) -- (1,1.3)
% node[midway,below] {$H$}
% -- (5,1.3)
% node[midway,below] {$m\in \mathcal M_0\cup\mathcal M_L$};

% \node at (4,0.5)
% {$\displaystyle Hm$};

% %------------------------------------------------
% % MLR decomposition
% %------------------------------------------------
% \node at (0,-1) {$\mathcal M_{LR}$};

% \draw[thick]
% (0,-1.5) -- (1,-1.5)
% node[midway,below] {$H$}
% -- (5,-1.5)
% node[midway,below] {$m\in \mathcal M_R\cup\mathcal M_{LR}$}
% -- (6,-1.5)
% node[midway,below] {$H$};

% \node at (4,-2.3)
% {$\displaystyle HmH$};

% \end{tikzpicture}

% \caption{First return decomposition for Motzkin paths associated with involutions avoiding $\{4321,3421,4312\}$.}
% \label{fig:first-return-3421}
% \end{figure}

\subsection{$I_n(4321,1342,1423)$}\label{1342}

\begin{thm}\label{connected_1423}
Let $\pi$ be a connected involution in $I_n(4321)$ and let
$m=\beta(\pi)$.

Then $\pi$ avoids $1423$ if and only if $m$ avoids all the subsequences 
$DU,$ $DHH,$ $HUU,$ $HUHH,$ $DHDD.$
This is the case if and only if $m$
is of one of the following forms.

\begin{enumerate}
    \item $H,$ 
    \item $U^aH^bUD^{a+1}$ with $a,b>0,$

    \item $U^aD^a$ with $a>0,$

    \item $U^aH^bUHD^{a+1}$ with $a,b>0,$
    \item $U^aH^bD^a$ with $a,b>0$
    \item $U^{a+1}D^{a}HD$ with $a>0,$
    \item $U^aH^bUD^aHD$ with $a,b>0,$
    \item $U^{a+1}H^bD^aHD$ with $a,b>0.$
\end{enumerate}
\end{thm}
\proof
First of all notice that if the connected path $m=\beta(\pi)$ contains one of the following patterns $DU,$ $DHH,$ $HUU,$ $HUHH,$ $DHDD$ then $\pi$ contains $1423.$

Moreover, the connected Motzkin paths avoiding all of these subpatterns are precisely those of the eight forms listed in the statement. 

It is also immediate to realize that each path of these forms corresponds to a connected involution avoiding $1423.$
\endproof

The characterization of general paths corresponding to 1423-avoiding involutions is now trivial.
\begin{thm}\label{1423_notc}
Let $\pi$ be an involution in $I_n(4321,1423).$ Then $\pi$ is not connected if and only if $\beta(\pi)$ is of the form $mm'$ where $m$ is a connected path of one of the eight forms listed in the previous theorem and $m'$ is obtained by concatenating connected components given by the factors  $H,$ $UD,$ $UHD.$
\end{thm}

As a consequence of the previous results we have $$F_T(x,y)=G_T(x,y)\cdot\left(\frac{1}{1-x-x^2y-x^3y^2}\right)+1$$
where
\begin{align*}
G(x,y)=&  x+\frac{x^2y-x^3y+x^3y^2+2x^5y^2-x^6y^2+3x^6y^3}{(1-x)(1-x^2)}=\\ & x+\frac{x^5 y^2}{(1-x^2)(1-x)}+\frac{(x^2 y)}{(1-x^2)}+\frac{x^3 y^2}{(1-x^2)(1-x)}\\ & +\frac{x^5 y^2}{(1-x^2)}+3 \frac{x^6 y^3}{(1-x^2)(1-x)}.
\end{align*}
is the generating function of the connected paths listed in Theorem \ref{connected_1423}.

\subsection{$I_n(4321,1432)$}\label{1432}

\begin{lem}
Let $\sigma$ be an involution and let $\sigma=\sigma_1\ldots \sigma_k$ be its decomposition in connected components. Then $\sigma\in I(4321,1432)$   if and only if $\sigma_1$ avoids 1432 and 4321 and the involution $\sigma_2\ldots\sigma_k$ avoids 321.
\end{lem}

\begin{thm}
 Let $\sigma\in I(4321)$ be a connected involution.  Then $\sigma$ avoids 1432 if and only if $\beta(\sigma)$ avoids the subsequences HUH and DUH.
\end{thm}
\proof
Suppose that $\beta(\sigma)$ contains $H_1UH_2$. Let $\hat U$ be the leftmost U step preceding $H_2$ and $\hat D$ the D step associated with $\hat U.$ Since the Motzkin path $\beta(\sigma)$ is connected, $\hat D$ must lie to the right of $H_2.$ Thus the sequence of steps $H_1,\hat U,H_2,\hat D$ corresponds in $\sigma$ to the pattern 1432. If the path    $\beta(\sigma)$ contains $DUH$ it is possible to prove similarly that $\sigma$ contains the pattern 1432. 

Conversely, suppose that $\sigma$ contains the subsequence $adcb$ order isomorphic to $1432.$ The steps corresponding to those integers can be either $UUHD,$ or $HUHD,$ or $DUHD.$ 

In the first case, it is easily checked that the path contains either $HUHD,$ or $DUHD$ elsewhere. 

This concludes the proof. 
\endproof

We want to determine the generating function $G_T(x,y)$ of the connected involutions in $I(4321,1432)$ counted by length and number of descents. By the previous Theorem we have $$G_T(x,y)=\sum_{d\in \mathcal M''}x^{|d|}y^{\operatorname{wp}(d)},$$
where $\mathcal M''$ is the set of connected  Motzkin paths avoiding $DUH$ and $HUH.$

First of all notice that each connected Motzkin path avoiding DUH and HUH is of the form $UmD$ where $m$ is any path avoiding DUH and HUH. 

We will analyze separately the cases in which the Motzkin paths start or end with a horizontal step (or both), since elevating these paths gives rise to additional weak peaks.

Each Motzkin path avoiding DUH and HUH is of one of the following types:
\begin{itemize}
    \item the empty path,
    \item $H^tc,$ where $t\geq 0$ and $c$ is a Dyck path,
    \item $Um_1Dc,$ where $m_1$ is itself a Motzkin path avoiding DUH and HUH  and $c$ is a Dyck path,
    \item $UbDH^tc,$ where $b$ is a Motzkin path avoiding $DU$ and $HU$, $t\geq 0$ and $c$ is a Dyck path.
\end{itemize}

Let $A_i(x,y)$ ($A_f(x,y),$ $A_{if}(x,y),$ and $A_0(x,y),$ respectively)  be the generating function of the paths avoiding DUH and HUH that start but do not end (end but not start, start and end, do not start nor end, respectively) with an H step (notice that the single step $H$ is counted by $A_{if}$ and the empty path is counted by $A_0$ only).

Define $B_s(x,y)$ for $s\in\{i,f,if,0\}$ in the same way for the family of paths avoiding DU and HU.

First of all observe that a Motzkin path $m$ avoiding DU and HU is either empty, or ends with $H,$ or is of the form UmD, where m itself avoids DU and HU. 
As a consequence we have,
\begin{align*}
   A_i(x,y)=& x (\operatorname{Nar}(x^2,y)-1), \\
   A_f(x,y)=&\frac{x^3(yB_f(x,y)+y^2B_{if}(x,y)+B_0(x,y)-1+y)}{1-x},\\
   A_{if}(x,y)=&\frac{x}{1-x},\\
   A_0(x,y)=&1+A_f(x,y)(\operatorname{Nar}(x^2,y)-1)+\\ &x^2(A_0(x,y)+yA_i(x,y)+yA_f(x,y)+y^2A_{if}(x,y)-1+y)\operatorname{Nar}(x^2,y),\\
   B_i(x,y)=& 0,\\
   B_f(x,y)=&\frac{x(B_0(x,y)-1)}{1-x}\\
   B_{if}(x,y)=&\frac{x}{1-x},\\
   B_0(x,y)=& 1+x^2(B_0(x,y)-1+y+yB_f(x,y)+y^2B_{if}(x,y)),\\   
\end{align*}

where $$\operatorname{Nar}(x,y)=1+xy+(y+y^2)x^2\ldots$$

$$ 
=\frac{
1 +x- xy - \sqrt{\left(1-x(1+y)\right)^2-4yx^2}
}{
2x
}
$$

is the Narayana generating function that counts Dyck paths according to semilength ($x$) and number of peaks ($y$).

Now, the generating function for the connected involutions in $T,$ can be written as 
\begin{align*}    G_T(x,y)=&x+x^2(A_0(x,y)+yA_i(x,y)+yA_f(x,y)+y^2A_{if}(x,y)-1+y).
\end{align*}

Thanks to the previous system of equations, we get 

$$
\begin{aligned}
G_T(x,y)
={}&
\frac{
N x^6-(2N+1)x^4+(N+2)x^2-1
}{
N x^5-N x^4-(N+1)x^3+(N+1)x^2-(N x^5-x^3)y+x-1
}
\\[1ex]
&+
\frac{
(x^5-x^3)y^2
-\left(Nx^6+(N+1)x^5+(N-3)x^3-2x^4+x^2\right)y
}{
N x^5-N x^4-(N+1)x^3+(N+1)x^2-(N x^5-x^3)y+x-1
}-1
\end{aligned}
$$

with 

$$N=\operatorname{Nar}(x^2,y).$$

Thus,
$$F(x,y)=1+G_T(x,y)\Lambda(x,y),$$

where

$$\Lambda(x,y)=\frac{1}{1-x-x^2(N-1+y)}.$$

\subsection{$I_n(4321,2143)$}\label{2143}

It is obvious that an involution $\pi$ avoiding the non-connected pattern 2143 can have at most one connected component distinct from $H.$

The procedure \texttt{guess[T]}  suggests to try to prove that the connected paths corresponding to the involutions avoiding such patterns are $\mathcal W^{20}_4(T)=\{DU,HUDH\}.$ It is in fact the case, as proved in the following theorem, where we consider directly the whole set of paths, not only the connected one.  

\begin{thm}
An involution $\pi\in I_n(4321)$ avoids the pattern $2143$ if and only if $\beta(\pi)$ avoids $DU$ and $UHUDHD.$ This is the case if and only if $\beta(\pi)$
is of the form $H^{a_1}U^{b_1}\ldots H^{a_k}U^{b_k}H^{a_{k+1}}D^tH^s$ or of the form $H^sU^tH^{a_1}D^{b_1}\ldots H^{a_k}D^{b_k}H^{a_{k+1}}$ for some $a_i,b_i,s,t\in \mathbb N$ with $t=b_1+b_2+\ldots +b_k.$
\end{thm}
\proof
We first show that in a Motzkin path corresponding to a $2143$-avoiding involution $\pi$, it is impossible to have a subsequence of the form
\[
D\,\cdots\,U.
\]

Indeed, a $D$ step closes a transposition $(x,y)$, while a subsequent $U$ step opens a new transposition starting at a later position. This would give an occurrence of 2143 in the corresponding involution.

As noted above, in $\beta(\pi)$ there is at most one connected component different from $H$ and, since the path avoid $DU,$ all $U$-steps must occur in an initial block or in a block separated only by $H$-steps, and all $D$-steps must occur in a final block or in a block separated only by $H$-steps.

Assume that the path contains a subsequence of the form \[
U\cdots H\cdots U\cdots D \cdots H \cdots D.
\] Thus we can consider a  subsequence 
\[
U_0\cdots H\cdots U_l\cdots D_0 \cdots H \cdots D_l
\]

where $U_0$ and $D_0$ are the first up and down steps in the path and $U_l$ and $D_l$ are the last up and down steps.

This subsequence corresponds to an occurrence of the pattern  426153 in the involution $\pi.$ This pattern contains 2143.

Conversely, for any path $d$ of one of these forms, the involution $\beta^{-1}(d)$ avoids 2143.

This completes the proof.

\endproof

Let $\mathcal L$ be the set of Motzkin path of the form $H^{a_1}U^{b_1}\ldots H^{a_k}U^{b_k}H^{a_{k+1}}D^tH^s$ and $\mathcal R$ be the set of paths of the form $H^sU^tH^{a_1}D^{b_1}\ldots H^{a_k}D^{b_k}H^{a_{k+1}}.$ Set $\mathcal C=\mathcal L\cap \mathcal R.$ 
Let $\mathcal L'$ be subset of $\mathcal L$ consisting of the paths with either $s=0$ or $t=0.$

Set $$L(x,y)=\sum_{m\in \mathcal L}x^{|m|}y^{\operatorname{wp}(m)}$$ and define $L'(x,y),R(x,y),R'(x,y),C(x,y)$ analogously and notice that $L(x,y)=R(x,y)$ and $L'(x,y)=R'(x,y).$

By the first return decomposition it follows that 
\begin{align*}
 L'(x,y)= & \frac{1}{1-x}+\frac{x^2y}{1-x}+\frac{x^3y^2}{(1-x)^2}+\frac{x^3y}{1-x}\cdot \left( L'(x,y)-\frac{1}{1-x}\right)  \\ & +\frac{x^2}{1-x}\cdot \left( L'(x,y)-x L'(x,y)-1\right),
\end{align*}

from which we can deduce

$$ L'(x,y)=\frac{1-x+x^2y-2x^3y+x^3y^2-x^2+x^3}{(1-x)\cdot (1-x-x^2+x^3-x^3y)},$$

whence,

$$ L(x,y)=\frac{1}{1-x}+\left( L'(x,y)-\frac{1}{1-x}\right)\cdot \frac{1}{1-x},$$

$$ C(x,y)=\frac{1}{1-x}+\frac{x^2y^2}{(1-x)^2\cdot (1-x^2)}+\frac{x^3y^2}{(1-x)^3\cdot (1-x^2)},$$
and
$$ F_T(x,y)=2 L(x,y)- C(x,y).$$

% and we get

% $$F_T(x,y)=\frac{x^6y^3 - x^7 + 3x^6 - x^5 - 5x^4 + 5x^3 - (x^4 - x^3)y^2 + x^2 + (x^7 - 3x^6 + 2x^5 + 2x^4 - 3x^3 + x^2)y - 3x + 1}{x^8 - 4x^7 + 4x^6 + 4x^5 - 10x^4 + 4x^3 + 4x^2 - (x^8 - 3x^7 + 2x^6 + 2x^5 - 3x^4 + x^3)y - 4x + 1}.$$
\endproof

\subsection{$I_n(4321,2413,3142)$}\label{2413}

Also in this case, the set of paths is very simple. 
% The connected components are precisely the connected paths avoiding the subsequences $DH,\ DU$ and $HU.$ In other terms, the connected components do not have weak valleys. As a consequence, an arbitrary path corresponding to involutions of this kind is a concatenation of paths of the form $U^aH^kD^a,$ $a,k\geq 0.$
Since $2413$ and $3142$ are connected, by Remark~\ref{conn_patt} it is
enough to characterize the connected involutions of this class.

\begin{lem}\label{lem:2413}
A connected involution $\pi\in I_n(4321)$ avoids $2413$ and $3142$ if
and only if $\beta(\pi)$ has no weak valleys, namely, if and only if
$\beta(\pi)=H$ or $\beta(\pi)=U^aH^kD^a$ with $a\ge1$, $k\ge0$.
\end{lem}

\begin{proof}
A connected path with no factors $DU$, $DH$, $HU$ consists of a single $H$ step, or is of the form $U^aH^kD^a$ with $a\ge1$.

% Recall that, since $\pi$ avoids $4321$, every down step closes the
% transposition that was opened earliest among the active ones.

Suppose that $\beta(\pi)$ has a weak valley formed by the steps in
positions $i$ and $i+1$. Since the step in position $i$ is not
$U$ and the step in position $i+1$ and the path is connected, in $\pi$ there is a transposition $(p\ q)$  with $p<i<i+1<q$. We distinguish three cases according to the type of the weak valley.
\begin{itemize}
\item $DU$: the step $D$ closes a transposition $(p'\ i)$ and the step
$U$ opens a transposition $(i+1\ q')$. Since $(p'\ i)$ is closed before
$(p\ q)$ and $(i+1\ q')$ is opened after it, we have $p'<p$ and
$q<q'$. The entries in positions $p<i<i+1<q$ are $q,p',q',p$, which form
an occurrence of $3142$.
\item $DH$: as before, $D$ closes $(p'\ i)$ with $p'<p$, and $i+1$ is a
fixed point. The entries in positions $p'<p<i<i+1$ are $i,q,p',i+1$,
which form an occurrence of $2413$.
\item $HU$: $i$ is a fixed point and $U$ opens $(i+1\ q')$ with
$q<q'$. The entries in positions $i<i+1<q<q'$ are $i,q',p,i+1$, which
form an occurrence of $2413$.
\end{itemize}

Conversely, if $\beta(\pi)=U^aH^kD^a$, then
\[
\pi=(a+k+1)\cdots(2a+k)\ (a+1)\cdots(a+k)\ 1\cdots a,
\]
namely, $\pi$ is the concatenation of three increasing sequences, each
consisting of values smaller than all the values of the previous one.
Every pattern contained in such a permutation has the same property,
while neither $2413$ nor $3142$ can be split into increasing blocks of
this kind. Hence $\pi$ avoids both patterns. The case $\beta(\pi)=H$ is
trivial.
\end{proof}

As a consequence, the paths corresponding to involutions in
$I_n(4321,2413,3142)$ are precisely the concatenations of paths of the
form $H$ and $U^aH^kD^a$, $a\ge1$, $k\ge0$. 
Hence, we have

$$G_T(x,y)=x+\frac{x^2y}{1-x^2}+\frac{x^3y^2}{(1-x)(1-x^2)}$$

and 

$$F_T(x,y)=\frac{1}{1-G_T(x,y)}.$$

\subsection{$I_n(4321,3241,4213)$}\label{3241}

First of all, we state a Lemma whose proof is almost trivial but that will be useful in the following. 

\begin{lem}
    Let $\pi$ be an involution in $I_n(4321)$ and let $m=\beta(\pi)$  be the associated Motzkin path. If $\pi$ contains the subsequence $d,b,a,c$ with $a<b<c<d,$ then the four steps in $m$ corresponding to the entries $d,$ $b,$ $a$ and $c$ are either UHDH or UHDD.
\end{lem}
\proof
Notice that the three elements $d,b,a$ form a decreasing sequence, hence the corresponding steps must be different. Straightforward considerations show that the only two possible cases are UHDH and UHDD.
\endproof

We point out that not every occurrence of UHDH or UHDD in $\beta(\pi)$ corresponds to an occurrence of 4213 in $\pi.$ Consider, for example, the path $\beta(\pi)=UUHDD$ where $\pi=45312,$ which does not contain 4213 at all.

An involution avoids 4213 if and only if each of its connected components avoids the same pattern. Thus we begin considering connected involutions. 

\begin{thm}
 Let $\pi$ be a connected involution in $I_n(4321)$ and let $m=\beta(\pi)$  be the associated Motzkin path. Then $\pi$ avoids $4213$ if and only if $m$ avoids HU and HDH.
\end{thm}
\proof
The proof consists in three steps
\begin{itemize}
    \item[\textbf{Step 1}] prove that if the connected involution $\pi$ avoids 4213 then $m$ avoids $HU,$ 
    \item[\textbf{Step 2}] prove that if $\beta(\pi)$ avoids $HU,$ none of the occurrences of the subsequence UHDD corresponds to an occurrence 4213 in $\pi,$ and
    \item[\textbf{Step 3}] prove that if $\beta(\pi)$ avoids $HU,$ the presence of an occurrence of the subsequence UHDH implies the existence of an occurrence of 4213 in $\pi.$ 
\end{itemize}
Notice that these three facts, together with the previous lemma, conclude the proof. 

\begin{itemize}
    \item[\textbf{Step 1}] By way of contradiction, let the path $m$ be connected and contain $\hat H U.$
Since the path is connected, there exists at least one up step preceding $\hat H.$ Consider the rightmost such up step $U_1$ and the leftmost up step $U_2$ following $\hat H.$ Denote by $D_1$ and $D_2$ the down steps corresponding to $U_1$ and $U_2,$ respectively. 
Both these down steps must be to the right of $U_2,$ otherwise the path would be not connected. 
Then the steps $U_1,H,U_2,D_1,D_2$ correspond to a subsequence of the permutation order isomorphic to 42513.
    \item[\textbf{Step 2}] Suppose that the path $m$ does not contain HU but contains the subsequence $\hat U \hat H \hat D \overline D.$
    Denote by $a$ and $b$ the integers corresponding to $\hat H$ and $\overline{D}$ in $\pi.$
    If this pattern was order isomorphic to 4213, then we would have $a<b$ and hence the up step corresponding to $\overline{D}$ would follow $\hat H,$ yielding a contradiction.   
    \item[\textbf{Step 3}] Suppose that the path $m$ does not contain $HU$ but contains the subsequence  $\tilde U \hat H \hat D \overline H.$ Then the path contains also the subsequence $\overline U \hat H \hat D \overline H,$ where $\overline{U}$  is the closest up step to the right of $\hat H,$ and thus the last up step in the path since $m$ does not contain HU. The step $\overline{U}$ corresponds to the last (down) step $\overline D$ of the path. Similarly, since $m$ avoids HU, the up step corresponding to $\hat D$ lies before $\hat H.$ Thus the steps $\overline U,$ $\hat H,$ $\hat D$ and  $\overline H$ correspond to 4213. 
\end{itemize}

\endproof

\begin{cor}
     Let $\pi$ be a connected involution in $I_n(4321)$ and let $m=\beta(\pi)$  be the associated Motzkin path. Then $\pi$ avoids $4213$ if and only if $m$ is obtained from a connected Dyck path by inserting a possibly empty factor of adjacent horizontal steps in any non-final node of the last descending run. 
\end{cor}
\proof
The proof is immediate.
\endproof

By the previous characterization, we need the generating function of
connected Dyck paths $D(x,s,y)$ counted by semilength ($x$), length of the last
descent ($s$) and number of peaks ($y$). This generating function can be
obtained directly from the corresponding one  for Dyck paths that are not
necessarily connected, $\Omega(x,s,y),$ found by Deutsch in~\cite{Deu}.

$$
\Omega(x,s,y)=1+\frac{ys\,\operatorname{Nar}(x,y)-ys}{y+(1-s)\,(\operatorname{Nar}(x,y)-1)},
$$

thus

$$D(x,s,y)=(\Omega(x,s,y)-1)xs+1+xsy.
$$

By the previous corollary the generating function for the connected involutions in $I(4321,3241,4213)$ (including the empty one), counted by length ($x$) and descents ($y$) can thus be found as 
$$
G_T(x,y)=D(x^2,1,y)+x+\left. \frac{\partial D(x^2,s,y)}{\partial s} \right|_{s=1}\cdot \frac{xy}{1-x}-1
$$

thus 

$$F_T(x,y)=\frac{1}{1-G_T(x,y)}.$$

\subsection{$I_n(4321,3412)$}\label{3412}

The enumeration of this family can be obtained from the results in~\cite{EGGE2004451}.

We use the following characterization, proved in~\cite{Ba6}.

\begin{lem}\label{struttura_3412}
Let $\pi$ be any involution. Then $\pi\in I_n(4321,3412)$ if and only if $\beta(\pi)$ has height at most one. 
\end{lem}

This result allows us to determine the descent distribution over the set $I_n(4321,3412).$

\begin{thm}
 $$F_T(x,y)=\frac{1-x}{(1-x)(1-x-x^2y)-x^3y^2}$$   
\end{thm}
\proof
By Lemma \ref{struttura_3412}, we deduce that the connected components of the path $\beta(\pi)$ are either $H$ or $UH^kD$ with $k\geq 0.$

Thanks to Lemma \ref{discese_picchi_deboli}, the conclusion follows trivially. 
\endproof

\section{Beyond forbidden subsequences}\label{beyond}

Recall that the procedure \texttt{guess[T]} suggests a set of forbidden
subsequences for the paths corresponding to a given family of
involutions. In some cases, every path in the family avoids these
subsequences, but not every path avoiding them belongs to the family:
further conditions are needed to characterize the family. In this section we treat such families of involutions.

% avoiding 4321 and another pattern such that the corresponding Motzkin paths cannot be described only by the avoidance of the subsequences produced with \texttt{guess[T]} but require additional constraints. 

\subsection{$I_n(4321,1324)$}\label{1324}

In this case $\mathcal W^{20}_4(T)=\{DHU,\
   DHDH,\
   DUDU,\
   HUHU\}.$ 
  This suggests that every connected involution of this family should correspond to a connected path avoiding every subsequence in $\mathcal{W}_2^{20}.$ As proved in Theorems \ref{connected_1324} and \ref{1324_c}, this is not a sufficient condition to characterize completely the connected paths corresponding to involutions avoiding 1324. The reason for this fact is easy to explain: the set of connected Motzkin paths corresponding to the involutions of this family is not closed by containment, thus it cannot be described in terms of avoidances.
   As an example, consider the involutions $42618375=\beta^{-1}(UHUDUDHD),$ which contains 1324, and $5\ 7\ 3\ 9\ 10\ 2\ 8\ 4\ 6=\beta^{-1}(UUHUDUDHDD),$ which avoids it even if the first path is contained in the second as a subsequence. Notice also that the first of these involutions contains 1324 even if the corresponding path avoids all the subsequences in $\mathcal W^{20}_4.$

\begin{thm}\label{1324_notc}
Let $\pi$ be an involution in $I_n(4321,1324).$ Then $\pi$ is not connected if and only if $\beta(\pi)$ is of the form $U^aH^bD^aH^cU^eH^fD^e$ with $a,b,c,f,e\geq 0$ and where at least two factors among $U^aH^bD^a,$ $H^c$ and $U^eH^fD^e$ are non-empty.
\end{thm}
\proof

Observe that, if $\beta(\pi)=mm'$ where $m$ and $m'$ are non-empty Motzkin paths, $\beta^{-1}(m)$ avoids 132 and $\beta^{-1}(m')$ avoids 213. Now the assertion follows by the characterization of Motzkin paths corresponding to 132 and 213-avoiding involutions under Biane's map (see~\cite{Ba6}).
\endproof

% Two weak valleys are said to be \emph{consecutive} if no other weak
% valley occurs between them; other steps may lie between them.

\begin{thm}\label{connected_1324}
Let $\pi$ be a connected involution in $I_n(4321)$ and let
$m=\beta(\pi)$.

If $\pi$ avoids $1324$ then every pair of consecutive weak
valleys of $m$ is of one of the following types.

\begin{enumerate}
    \item A weak valley of type $HU$ followed by a weak valley of type $DU$.
    In this case the two valleys belong to a factor
    \[
    \hat H U\,U^{a}H^{b}D^{c}\tilde D \overline{U},
    \]
    where $c< h$ and $h$ is the height of the distinguished step
    $\hat H$.

    \item A weak valley of type $DU$ followed by a weak valley of type $DH$.
    In this case the two valleys belong to a factor
    \[
    DU\,U^{a}H^{b}D^{c}D\hat H,
    \]
    where $a< h$ and $h$ is the height of the distinguished step
    $\hat H$.

    \item A weak valley of type $HU$ followed by a weak valley of type $DH$.
    In this case the two valleys belong to a factor
    \[
    \hat H U\,U^{a}H^{b}D^{c}D\tilde H,
    \]
    where $a< h$ and $h$ is the height of $\tilde H$.
    Equivalently, one may require $c< h'$, where $h'$ denotes the
    height of $\hat H$.
\end{enumerate}
\end{thm}

\begin{proof}
Assume first that $\pi$ avoids $1324$.

We begin by excluding several configurations of consecutive weak valleys.

Suppose that $m$ contains two consecutive weak valleys of type $DH$.
Then $m$ contains a factor of the form
\[
\hat D\,\hat H\,H^{a}\,\tilde D\,D^{b}\,\tilde H,
\qquad a,b\ge0.
\]
The four steps
\[
\hat D,\ \hat H,\ \tilde D,\ \tilde H
\]
correspond, under Biane's bijection, to four entries forming an
occurrence of the pattern $1324$, a contradiction.

By the same argument one sees that $m$ cannot contain two consecutive
weak valleys both of type $HU$, nor a pair of consecutive weak valleys
of types $DH$ and $DU$, nor a pair of types $DU$ and $HU$.

Next suppose that two consecutive weak valleys are of types $DH$ and
$HU$. Then $m$ contains a factor
\[
\hat D \hat H H^{k}U,
\qquad k\ge0,
\]
whose horizontal steps are not at ground level.

Since $m$ is connected, this factor must be followed by at least two
down steps: the first down step $D_1$ occurring after the factor and the
down step $D_2$ matched with the final up step of the factor.
The steps
\[
\hat D,\ \hat H,\ D_1,\ D_2
\]
give rise to an occurrence of $1324$ in the corresponding involution,
which is impossible.

Finally, suppose that $m$ contains two consecutive valleys of type $DU$.
Then $m$ contains a factor
\[
DU^{a}H^{b}D^{c}U .
\]
The first two down steps and the last two up steps of this factor
correspond to an occurrence of the pattern $1324$, again a contradiction.

Therefore the only possible pairs of consecutive weak valleys are those
listed in the statement. 

The inequalities on the parameters are exactly
the conditions preventing the constructions above from producing an
occurrence of $1324$. In fact, suppose that, in Case 1, $c\ge h.$ Then, the step $\tilde U$ associated with $\tilde D$ is to the right of $\hat H.$ Hence $\hat H,$ $\tilde U,$ $\tilde D$ and $\overline U$ correspond to an occurrence of 1324 in $\pi.$ Similar considerations hold in the last two cases. 
\end{proof}

\begin{thm}\label{1324_c}
Let $\pi \in I_n(4321)$ be connected. Then $\pi$ avoids $1324$ if and only if $\beta(\pi)$ is connected and
\begin{itemize}
    \item has at most one weak valley,
    \item has exactly two weak valleys and is of one of the following forms 
    $$U^{a+c}H^eU^bH^fD^aH^gD^{b+c}$$ with $a,b,e,g\geq 1$ and $c,f\geq 0$ or
   $$U^{a+b+c}H^fD^bU^{1+e}H^gD^{1+c}H^iD^{a+e}$$ with $a,b,i\geq 1$ and $c,e,f,g\geq 0$ or its symmetric,
   \item has exactly three weak valleys and is of the form 
   $$U^{a+b+c}H^gU^eH^iD^aU^fH^mD^bH^pD^{e+f+c}$$ with $a,b,e,f,g,p\geq 1$ and $c,i,m\geq 0.$
\end{itemize}
\end{thm}
\proof
Let $\pi\in I_n(4321,1324)$ be a connected involution. By the previous theorem, it follows that the connected path $\beta(\pi)$ has at most three weak valleys, since pairs of consecutive weak valleys are constrained to be either $HU-DU,$ or $DU-DH,$ or $HU-DH.$ Taking into account the conditions of the previous theorem, it is immediate to realize that the only possible forms are those given in the statement. 

Conversely, it is straightforward to check that any connected path of the form given in the statement corresponds to an involution avoiding 1324. 
\endproof

\begin{thm}
We have $$F_T(x,y)=A(x,y)/B(x)$$
where
$$B(x)=(1+x)^5(1-x)^7$$ and 
\begin{align*}
A(x,y)=& (y^4 - 4y^3 + 4y^2 - y)x^{12} + (y^4 - 2y^2 + 1)x^{11}\\ & - (3y^4 - 13y^3 + 12y^2 - 5y + 1)x^{10} - (3y^4 - 3y^2 + 5)x^9 \\ & + (2y^4 - 12y^3 + 14y^2 - 10y + 5)x^8 + (2y^4 - 2y^3 + y^2 + 10)x^7 \\ & + (y^4 + 3y^3 - 8y^2 + 10y - 10)x^6 + (2y^3 - 3y^2 - 10)x^5 \\ & + (2y^2 - 5y + 10)x^4 + (y^2 + 5)x^3 + (y - 5)x^2 - x + 1.
\end{align*}
\end{thm}
\proof
The result follows from Theorems~\ref{1324_notc} and~\ref{1324_c} by using arguments similar to those employed in the proof of Theorem~\ref{GF_1243}.
\endproof

\subsection{$I_n(4321,2341,4123)$}\label{2341}

In this case, it is easy to see that the connected Motzkin paths corresponding to connected involutions in $I(4321,2341,4123)$ avoid the subsequences %$UUU$ e $HHHH.$
in $\mathcal W_4^{20}=\{DU,\ DDD,\
   DDH,\
   DHH,\
   HHU,\
   HUU,\
   UUU,\
   HHHH\}.$

However the converse cannot be true for a simple reason. The set of connected Motzkin paths corresponding to involutions avoiding $4321$ and $2341$ are not closed under
taking connected subsequences. For example, the involution  $52341=\beta^{-1}(UHHHD)$ contains 2341 (even if the path avoids all the listed subsequences) while  $5274163=\beta^{-1}(UHUHDHD)$ does not contain it, even if $UHHHD$ is contained in $UHUHDHD.$ This shows that this set of connected paths cannot be defined through avoidances. 

Nonetheless, the fact that the connected paths corresponding to $I_n(4321,2341,4123)$ avoid all the subsequences in $\mathcal W_4^{20},$ is sufficient to conclude that such
 paths must have a length not exceeding 7. 

The corresponding not-necessarily-connected paths are arbitrary sequences of connected components.  
This gives 

$$F_T(x,y)=\frac{1}{1-G_T(x,y)},$$

where the generating function for the connected components is, by a direct inspection,

$$G_T(x,y)=x+x^2y+x^3y^2+x^4y^2+x^4y+3x^5y^2+3x^6y^3+x^6y^2+x^7y^4.$$

\section{Exceptional classes}\label{exceptional}

In this section we consider the only two cases in which the procedure \texttt{guess[T]} applied to a single pattern does not produce any set of avoided subsequences. The reason, as we will see below, is that the corresponding set of connected Motzkin paths cannot be described in terms of avoidance of any subsequence. 

The \emph{height} of a horizontal step of a Motzkin path is the common
height of its two endpoints.

\subsection{$I_n(4321,3421,4312)$}\label{3421}

\begin{thm}
A permutation $\pi\in I_n(4321)$ contains the pattern $3421$ (and hence the pattern $4312$) if and only if $\beta(\pi)$ contains at least one $H$ step at height greater than one. 
\end{thm}
\proof
Let $\pi\in I_n(4321)$ be a permutation containing a subsequence $c,d,b,a$ with $a<b<c<d$. Let $S_1,S_2,S_3, S_4$ be the steps corresponding to $c,d,b$ and $a,$ respectively, in the Motzkin path $\beta(\pi).$  Straightforward considerations based on the definition of the Biane's map show that we must have $S_1=S_2=U,$ $S_3=H$ and $S_4=D.$ 

% Indeed, the entries $c$ and $d$ must correspond to the opening of two transpositions, while $b$ must be a fixed point and $a$ must close the most recently opened transposition. Since $\pi$ avoids $4321$, all labels in $\beta(\pi)$ are equal to $1$, hence the two transpositions corresponding to $c$ and $d$ are properly nested.

Consequently, when the horizontal step $S_3$ corresponding to $b$ is read, the two transpositions opened by $S_1$ and $S_2$ are still active. Therefore $S_3$ is a horizontal step at height at least $2$.

Conversely, suppose that $\beta(\pi)$ contains a horizontal step $H$ at height greater than $1$. Let $b$ be the fixed point corresponding to this horizontal step. Since the height is at least $2$, at least two transpositions are open when $b$ is read. Let
\[
(x,y)\qquad\text{and}\qquad(z,t)
\]
be the corresponding transpositions, with
\[
x<z<b<y<t.
\]
Thus the symbols $y,t,b,x$ appear in the permutation in this order and it
is precisely an occurrence of the pattern $3421$.

\endproof

Let $\mathcal M'$ be the set of Motzkin paths that do not contain $H$ steps at height greater than one. Let $\mathcal M_0$ (resp. $\mathcal M_L$, $\mathcal M_R$, $\mathcal M_{RL}$) be the subset of $\mathcal M'$ consisting of the paths that neither start nor end (resp.  start but do not end, end but do not start, both start and end) with an $H$  step.

Set $$M_0(x,y)=\sum_{m\in \mathcal{M}_0}x^{|m|}y^{\wep(m)}$$
and define $M_L,$ $M_R$ and $M_{LR}$ analogously.  Notice that $M_L=M_R$ by symmetry. 

By the first return decomposition we obtain the following relationships  

$$M_0(x,y)=1+x^2(N-1)(M_0(x,y)+M_L(x,y))+x^2y(M_0(x,y)+M_L(x,y)),$$

$$M_L(x,y)=x(M_0(x,y)+M_L(x,y)-1), $$

$$M_{LR}(x,y)=x(M_R(x,y)+M_{LR}(x,y))+x,$$

with $N=\operatorname{Nar}(x^2,y),$
where $\operatorname{Nar}(x,y)$ is the Narayana generating function (se Subsection \ref{1432}).
%$$\operatorname{Nar}(x,y)=\frac{1-x(1+y)-\sqrt{(1-x(1+y))^2-4yx^2}}{2x}+1$$
% is the  generating function of Narayana numbers, namely, the generating function of Dyck paths counted by semilength ($x$) and number of peaks ($y$).

Straightforward computations give 

$$M_0(x,y)=\frac{1-x-x^3y-x^3(N-1)}{1-x-x^2y-x^2(N-1)},$$

% $$M_R(x,y)=M_L(x,y)=\frac{xM_0(x,y)-x}{1-x}, $$

%M_R=(x^3y+x^3(N-1))}{(1-x-x^2y-x^2(N-1))

%$$M_{LR}=\frac{xM_R(x,y)+x}{1-x},$$

$$M_L(x,y)=M_R(x,y)=\frac{(N-1)x^3+x^3y}{1-x-x^2y-x^2(N-1)},$$

$$M_{LR}(x,y)=\frac{x-(N-1)x^3-x^3y}{1-x-x^2y-x^2(N-1)},$$

and, applying once more the first return decomposition, we get

$$F_T(x,y)=\frac{1}{1-x-2x^2y M_R(x,y)-x^2y^2M_{LR}(x,y)-x^2y-x^2(M_0(x,y)-1)}.$$

\subsection{$I_n(4321,4231)$}
\label{4231}

Two horizontal steps of a path are
\emph{consecutive} if no horizontal step lies between them. 

\begin{defn}\label{def:admissible}
A Motzkin
path $m$ is \emph{admissible} if, for every pair of consecutive
horizontal steps of $m$, the number of down steps lying between them is
at least equal to the height of the first of the two. 
\end{defn}

\begin{thm}\label{thm:4231}
 Let $\pi\in I(4321).$ Then $\pi$ avoids 4231 if and only if $\beta(\pi)$ is admissible.    
\end{thm}
\proof
% Since $\pi$ avoids $4321$, all labels of $\beta(\pi)$ equal $1$, so each
% down step closes the transposition that was opened earliest among the
% active ones.
Suppose that $\beta(\pi)$ is not admissible: there are two consecutive horizontal steps $H_1$ and $H_2$ at positions $f<f'$ such that $H_1$ has height $h$ and fewer than $h$
down steps lie between them. The $h$ transpositions active at $f$ are
closed, in the order in which they were opened, by the first $h$ down
steps following $H_1$; hence at least one of them, say $(i\ j)$ with
$i<f<f'<j$, is still active at $f'$. The entries $j,f,f',i$ in positions
$i,f,f',j$ form an occurrence of $4231$.

Conversely, let $v_1v_2v_3v_4$ be an occurrence of $4231$ at positions
$p_1<p_2<p_3<p_4$. Both $v_1v_2v_4$ and $v_1v_3v_4$ are decreasing, so,
by Remark~\ref{oss:decreasing} and a direct check of the possible cases,
the steps corresponding to $v_1,v_2,v_4$ (and to $v_1,v_3,v_4$) are
$U,H,D$ in this order. Hence $p_2$ and $p_3$ are fixed points, and the
transposition $(p_1\ v_1)$ satisfies $p_1<p_2<p_3=v_3<v_1$. Thus this
transposition is active at two consecutive fixed points $f<f'$ with
$p_2\le f<f'\le p_3$. As shown above, this happens only if fewer than
$h$ down steps lie between $f$ and $f'$, where $h$ is the height of
$f$. Therefore $\beta(\pi)$ is not admissible.
\endproof

We denote by $\mathcal A$
the set of admissible Motzkin paths, so, by Lemma~\ref{discese_picchi_deboli}, we have
\[
F_T(x,y)=\sum_{m\in\mathcal A}x^{|m|}\,y^{\wep(m)} .
\]

The aim of this subsection is to determine a functional equation that
characterizes $F_T(x,y)$ (Theorem~\ref{thm:FE4231}). The main idea is to
cut an admissible path at its horizontal steps: the pieces are words
over $\{U,D\}$, and the pieces that are subject to a constraint can be
counted in closed form.

\subsubsection{Notation and conventions}

In this subsection we recall some basic facts and introduce some notations useful in the following. 

First of all, recall that, if
$f(u)=\sum_{n\ge0}f_n(u)x^n\in\mathbb Q[y,u][[x]]$ and $g$ is a formal
power series in $x$ with zero constant term (whose coefficients lie in
$\mathbb Q[y]$ or in $\mathbb Q[y,u]$), then the substitution
\[
f(g)=\sum_{n\ge0}f_n(g)\,x^n
\]
 is well defined. 
 % since the $n$-th summand has valuation at least $n.$

Throughout this subsection we use the following notation:
\begin{align*}
K(u)&=-xu^2+\bigl(1+(1-y)x^2\bigr)u-x,\\[1mm]
\theta(u)&=1+\frac{yxu}{1-xu}=\frac{1-(1-y)xu}{1-xu},\\[1mm]
\omega(u)&=x\,\theta(u)=\frac{x\bigl(1-(1-y)xu\bigr)}{1-xu},
\end{align*}
and we denote by $u_0=u_0(x,y)$ the unique formal power series in $x$
with zero constant term such that $K(u_0)=0$, namely
\[
u_0=\frac{1+(1-y)x^2-\sqrt{\bigl(1+(1-y)x^2\bigr)^2-4x^2}}{2x}
=x+yx^3+O(x^5).
\]
Finally, we set
\[
\lambda(u)=\frac{x\bigl(1-(1-y)xu\bigr)}{1-xu_0}.
\]

\begin{lem}\label{lem:kernel-identities}
We have
\[
K(u)=(1-xu)\bigl(u-\omega(u)\bigr).
\]
Consequently $\omega(u_0)=\lambda(u_0)=u_0$.
\end{lem}

\begin{proof}
We have
$(1-xu)\bigl(u-\omega(u)\bigr)=u-xu^2-x+(1-y)x^2u=K(u)$.
Since $1-xu_0$ is invertible and $K(u_0)=0$, we get $u_0=\omega(u_0)$.
Moreover $\lambda(u_0)=x\bigl(1-(1-y)xu_0\bigr)/(1-xu_0)=\omega(u_0)=u_0$.
\end{proof}

We point out that formally we can write $K(u)=-x(u-u_0)(u-u_0^{-1}).$

\subsubsection{Words over $\{U,D\}$ and their weights}

Given any word $w$ and a letter $L,$ denote by $\#L(w)$ the number of occurrences of $L$ in $w.$

Let $w=s_1\cdots s_\ell$ be a word over $\{U,D\}$ and let $a\ge0$.
We regard $w$ as a lattice path starting at a certain height $a$; its
\emph{final height} is $e_a(w)=a+\#U(w)-\#D(w)$, and we say that $w$
is \emph{nonnegative from $a$} if it starts at height $a$ and if every prefix of $w$ ends at a
nonnegative height, namely if the path never goes below the $x$-axis.
We define the \emph{weight} of $w$ as
\[
\operatorname{wt}(w)=\#\{\text{factors } UD \text{ of } w\}
+\chi(w\text{ starts with }D)+\chi(w\text{ ends with }U),
\]
where $\chi(P)$ equals $1$ if the property $P$ holds and $0$ otherwise.
The relevance of this definition is explained by the following lemma.

\begin{lem}\label{lem:weight-segments}
Let $m$ be a Motzkin path and write
\[
m=w_0\,H\,w_1\,H\cdots H\,w_k,\qquad k\ge0,
\]
where $w_0,\dots,w_k$ are (possibly empty) words over $\{U,D\}$.
Then
\[
\wep(m)=\sum_{i=0}^{k}\operatorname{wt}(w_i).
\]
\end{lem}

\begin{proof}
A weak peak of $m$ is a factor of length two of $m$ equal to $UD$, $UH$
or $HD$. A factor of length two of $m$ either lies inside some $w_i$, or
contains at least one of the displayed horizontal steps. In the first
case it is a weak peak if and only if it is a factor $UD$ of $w_i$.
In the second case it is either $HH$ (when some $w_i$, $1\le i\le k-1$,
is empty), or it consists of the last step of a nonempty $w_{i-1}$
followed by $H$, or of $H$ followed by the first step of a nonempty
$w_i$. The factor $HH$ is not a weak peak; the factor given by the last step of
$w_{i-1}$ and the subsequent $H$ is a weak peak if and only if $w_{i-1}$ ends with $U$;
the factor given by $H$ followed by the first step of $w_i$ is a weak peak if and only if
$w_i$ starts with $D$. Hence
\[
\wep(m)=\sum_{i=0}^k \#\{UD\text{ in }w_i\}
+\sum_{i=0}^{k-1}\chi(w_i\text{ ends with }U)
+\sum_{i=1}^{k}\chi(w_i\text{ starts with }D).
\]
Finally, $w_0$ starts at height $0$, so it cannot start with $D$, and
$w_k$ ends at height $0$, so it cannot end with $U$. Therefore the two last sums can be extended
to $0\le i\le k$, and the claim follows.
\end{proof}

For every $a\ge0$ we introduce the following generating functions:
\begin{align*}
W_a(u)&=\sum_{w}x^{|w|}\,y^{\operatorname{wt}(w)}\,u^{e_a(w)},
&&w \text{ nonnegative from } a,\\
S_a(u)&=\sum_{w}x^{|w|}\,y^{\operatorname{wt}(w)}\,u^{e_a(w)},
&&w \text{ nonnegative from } a \text{ with } \#D(w)\ge a,\\
C_a(u)&=\sum_{w}x^{|w|}\,y^{\operatorname{wt}(w)}\,u^{e_a(w)},
&&w \text{ starts at height }a\text{ and } \#D(w)< a.
\end{align*}
We call the words counted by $S_a$ \emph{good} and the words counted by
$C_a$ \emph{bad}. Notice that $C_0=0$ and $S_0=W_0$. The following
simple observation is the key to the whole computation.

\begin{lem}\label{lem:good-bad}
For every $a\ge1$ we have $S_a(u)=W_a(u)-C_a(u)$.
\end{lem}

\begin{proof}
Let $w$ be a word with $d<a$ down steps. Every prefix of $w$ contains
at most $d$ down steps, hence it ends at height at least
$a-d\ge1$. Therefore every bad word is nonnegative from $a$, and the
words nonnegative from $a$ split into good words and bad words.
\end{proof}

\subsubsection{Nonnegative words}

\begin{lem}\label{lem:Wa}
For every $a\ge1$ we have
\[
W_a(u)=\theta(u)\Bigl(u^a+R_a(u)\Bigr),
\qquad\text{where}\qquad
R_a(u)=\frac{xy}{K(u)}\left(u^a-\frac{(1-xu)\,u_0^{\,a}}{1-xu_0}\right).
\]
Moreover,
\[
W_0(u)=1-y+\frac{y}{K(u)}\bigl(u-\lambda(u)\bigr).
\]
\end{lem}

\begin{proof}
Fix $a\ge0$. Let $R^U_a(u)$ (respectively, $R^D_a(u)$) be the
generating function of the nonempty words nonnegative from $a$ ending
with $U$ (respectively, with $D$), where each word $w$ is weighted by
\[
x^{|w|}\,y^{\#\{UD\text{ in }w\}+\chi(w\text{ starts with }D)}\,u^{e_a(w)},
\]
namely without the contribution $\chi(w\text{ ends with } U)$. Since
the empty word contributes $u^a$ and a word ending with $U$ gets an
extra factor $y$ in $W_a$, we have
\begin{equation}\label{eq:W-in-terms-of-R}
W_a(u)=u^a+y\,R^U_a(u)+R^D_a(u).
\end{equation}
Every nonempty word $w$ can be uniquely written as $w=w's$, where $s$
is its last step and $w'$ is nonnegative from $a$.

\smallskip\noindent
\emph{Appending $U$.} A step $U$ can be appended to every word $w'$,
it does not create any factor $UD$, and it does not change whether the
word starts with $D$ (if $w'$ is empty, the new word starts with $U$).
It multiplies the weight by $xu$. Hence
\begin{equation}\label{eq:RU}
R^U_a(u)=xu\bigl(u^a+R^U_a(u)+R^D_a(u)\bigr).
\end{equation}

\smallskip\noindent
\emph{Appending $D$.} A step $D$ can be appended to $w'$ if and only if
$e_a(w')\ge1$, and it multiplies the weight by $x/u$. It gives an
additional factor $y$ if $w'$ is empty (the new word starts with $D$)
or if $w'$ ends with $U$ (a new factor $UD$ is created), while it gives
no additional factor if $w'$ ends with $D$.

Assume first $a\ge1$. The words $w'$ ending at height $0$ are neither
the empty word (which ends at height $a\ge1$) nor words ending with
$U$ (which end at a positive height); hence they are exactly the words
counted by $R^D_a(0)$. We obtain
\begin{equation}\label{eq:RD}
R^D_a(u)=\frac{x}{u}\Bigl(y\,u^a+y\,R^U_a(u)+R^D_a(u)-R^D_a(0)\Bigr), \quad \text{for }a\geq1.
\end{equation}
From \eqref{eq:RU} we get $(1-xu)R^U_a=xu\,(u^a+R^D_a)$, hence
\[
y\,u^a+y\,R^U_a=\frac{y\,u^a+yxu\,R^D_a}{1-xu}.
\]
Substituting into \eqref{eq:RD} and multiplying by $u(1-xu)$ we get
\[
u(1-xu)R^D_a=xy\,u^a+x^2yu\,R^D_a+x(1-xu)R^D_a-x(1-xu)R^D_a(0),
\]
that is,
\begin{equation}\label{eq:kernel-RD}
K(u)\,R^D_a(u)=xy\,u^a-x(1-xu)\,R^D_a(0),
\end{equation}
since $u(1-xu)-x(1-xu)-x^2yu=K(u)$. 

Here we use the so called \textit{kernel method} (see e.g.~\cite{Prodinger2003}).

Identity \eqref{eq:kernel-RD}
holds in $\mathbb Q[y,u][[x]]$, so we can replace $u$ by $u_0$. Since
$K(u_0)=0$, we get
\[
R^D_a(0)=\frac{y\,u_0^{\,a}}{1-xu_0},
\]
and \eqref{eq:kernel-RD} gives $R^D_a(u)$ as the expression denoted by $=R_a(u)$ in the statement. Finally, by
\eqref{eq:W-in-terms-of-R} and \eqref{eq:RU},
\[
W_a=u^a+\frac{yxu\,(u^a+R^D_a)}{1-xu}+R^D_a
=\Bigl(1+\frac{yxu}{1-xu}\Bigr)\bigl(u^a+R^D_a\bigr)
=\theta(u)\bigl(u^a+R_a(u)\bigr).
\]

\smallskip
Assume now $a=0$. The only difference is that the empty word ends at
height $0$, so a $D$ step cannot be appended to it. Equation
\eqref{eq:RU} is unchanged (with $u^a=1$), while \eqref{eq:RD} becomes
\[
R^D_0(u)=\frac{x}{u}\Bigl(y\,R^U_0(u)+R^D_0(u)-R^D_0(0)\Bigr).
\]
The same computation as above yields
\[
K(u)\,R^D_0(u)=x^2yu-x(1-xu)\,R^D_0(0),
\]
and replacing $u$ by $u_0$ we obtain
$R^D_0(0)=xy\,u_0/(1-xu_0)$. Hence
\[
R^D_0(u)=\frac{x^2y}{K(u)}\left(u-\frac{(1-xu)\,u_0}{1-xu_0}\right),
\qquad
W_0(u)=\theta(u)\bigl(1+R^D_0(u)\bigr).
\]
It remains to put $W_0$ in the stated form. Using $x\theta=\omega$,
$\theta(1-xu)=1-(1-y)xu$ and $\theta=1-y+y/(1-xu)$ we get
\[
W_0(u)=1-y+\frac{y}{1-xu}
+\frac{y}{K(u)}\left(xu\,\omega(u)-\frac{x^2\bigl(1-(1-y)xu\bigr)u_0}{1-xu_0}\right).
\]
By Lemma~\ref{lem:kernel-identities},
$\frac{1}{1-xu}=\frac{u-\omega(u)}{K(u)}$; moreover
$\omega(u)(1-xu)=x\bigl(1-(1-y)xu\bigr)$. Therefore
\begin{align*}
W_0(u)-(1-y)
&=\frac{y}{K(u)}\left(u-\omega(u)(1-xu)-\frac{x^2\bigl(1-(1-y)xu\bigr)u_0}{1-xu_0}\right)\\
&=\frac{y}{K(u)}\left(u-x\bigl(1-(1-y)xu\bigr)\Bigl(1+\frac{xu_0}{1-xu_0}\Bigr)\right)\\
&
=\frac{y}{K(u)}\bigl(u-\lambda(u)\bigr),
\end{align*}
as claimed.
\end{proof}

\begin{oss}
Setting $u=0$ in the previous lemma, one obtains
$W_0(0)=1+\frac{xy\,u_0}{1-xu_0}=\operatorname{Nar}(x^2,y)$, (see Subsection \ref{1432}), as expected. For
$y=1$ one has $\theta(u)=\frac{1}{1-xu}$, $K(u)=u-x-xu^2$ and
$\frac{x}{1-xu_0}=u_0$, and the lemma gives the formula
$W_a(u)=\frac{u^{a+1}-u_0^{a+1}}{u-x-xu^2}$.
\end{oss}

\subsubsection{Bad words}

\begin{lem}\label{lem:Ca}
For every $a\ge1$ we have
\[
C_a(u)=\theta(u)\,u^a+\frac{y}{K(u)}\Bigl(\omega(u)\,u^a-u\,\omega(u)^a\Bigr).
\]
\end{lem}

\begin{proof}
A word with exactly $d$ down steps can be uniquely written as
\[
w=U^{m_0}\,D\,U^{m_1}\,D\cdots D\,U^{m_d},\qquad m_0,\dots,m_d\ge0.
\]
Its factors $UD$ correspond to the indices $0\le i\le d-1$ with
$m_i>0$; moreover, if $d\ge1$, then $w$ starts with $D$ if and only if
$m_0=0$, and $w$ ends with $U$ if and only if $m_d>0$.

If $d=0$, then $w=U^{m_0}$ and $\operatorname{wt}(w)=\chi(m_0>0)$.
These words contribute
\[
u^a\sum_{m\ge0}y^{\chi(m>0)}(xu)^m=u^a\Bigl(1+\frac{yxu}{1-xu}\Bigr)=\theta(u)\,u^a.
\]

If $d\ge1$, the weight of $w$ is the sum of the following
contributions:
\begin{itemize}
\item the block $U^{m_0}$ always contributes exactly $1$ to the count of weak peaks: if $m_0=0$
the word starts with $D$, while if $m_0>0$ the block is followed by
the factor $UD$;
\item each block $U^{m_i}$ with $1\le i\le d-1$ contributes $1$ if and
only if $m_i>0$, since in this case it is followed by the factor $UD$;
\item the block $U^{m_d}$ contributes $1$ if and only if $m_d>0$, since
in this case the word ends with $U$.
\end{itemize}
Since each $U$ step multiplies the weight by $xu$, each $D$ step by
$x/u$, and the starting height gives the factor $u^a$, the words with
$d\ge1$ down steps contribute
\[
u^a\cdot\frac{y}{1-xu}\cdot\theta(u)^d\cdot\Bigl(\frac{x}{u}\Bigr)^d
=\frac{y}{1-xu}\,u^a q^d,
\qquad\text{where } q=\frac{\omega(u)}{u}.
\]
Summing over $0\le d\le a-1$ we obtain
\[
C_a(u)=\theta(u)\,u^a+\frac{y}{1-xu}\sum_{d=1}^{a-1}u^a q^d .
\]
Now $(1-q)\sum_{d=1}^{a-1}q^d=q-q^a$, so that
\[
(1-q)\sum_{d=1}^{a-1}u^aq^d=u^{a-1}\omega(u)-\omega(u)^a .
\]
By Lemma~\ref{lem:kernel-identities},
$1-q=\frac{u-\omega(u)}{u}=\frac{K(u)}{u(1-xu)}$; hence
\[
\frac{y}{1-xu}\sum_{d=1}^{a-1}u^aq^d
=\frac{y}{1-xu}\cdot\frac{u(1-xu)}{K(u)}\bigl(u^{a-1}\omega(u)-\omega(u)^a\bigr)
=\frac{y}{K(u)}\bigl(\omega(u)u^a-u\,\omega(u)^a\bigr),
\]
and the claim follows. (For $a=1$ the sum is empty and the formula
correctly gives $C_1(u)=\theta(u)\,u$.)
\end{proof}

\subsubsection{Good words}

\begin{prop}\label{prop:Sa}
For every $a\ge0$ we have
\[
S_a(u)=\frac{y}{K(u)}\Bigl(u\,\omega(u)^a-\lambda(u)\,u_0^{\,a}\Bigr)+(1-y)\,\chi(a=0).
\]
\end{prop}

\begin{proof}
For $a=0$ we have $S_0=W_0$, and the claim is the second part of
Lemma~\ref{lem:Wa}. Let $a\ge1$. By Lemmas~\ref{lem:good-bad},
\ref{lem:Wa} and~\ref{lem:Ca},
\[
S_a(u)=\theta(u)\,R_a(u)-\frac{y}{K(u)}\Bigl(\omega(u)\,u^a-u\,\omega(u)^a\Bigr).
\]
Since $x\theta(u)=\omega(u)$ and $\theta(u)(1-xu)=1-(1-y)xu$, we have
\[
\theta(u)\,R_a(u)=\frac{y}{K(u)}\left(\omega(u)\,u^a
-\frac{x\bigl(1-(1-y)xu\bigr)}{1-xu_0}\,u_0^{\,a}\right)
=\frac{y}{K(u)}\Bigl(\omega(u)\,u^a-\lambda(u)\,u_0^{\,a}\Bigr),
\]
and the terms $\omega(u)u^a$ cancel.
\end{proof}

The crucial feature of Proposition~\ref{prop:Sa} is that the starting
height $a$ appears only as an exponent. As a consequence, appending a
good word to a family of paths amounts to a substitution, as the next
corollary shows.

\begin{cor}\label{cor:operator}
Let $f(u)=\sum_{a\ge0}f_a\,u^a\in\mathbb Q[y,u][[x]]$, with
$f_a\in\mathbb Q[y][[x]].$
% and assume that the valuation of $f_a$ tends
% to infinity as $a\to\infty$.
Then
\[
\sum_{a\ge0}f_a\,S_a(u)=\frac{y}{K(u)}\Bigl(u\,f\bigl(\omega(u)\bigr)-\lambda(u)\,f(u_0)\Bigr)+(1-y)\,f(0).
\]
\end{cor}

\begin{proof}
% The series on the left-hand side converges, since $S_a(u)$ has
% nonnegative valuation.
The claim follows from
Proposition~\ref{prop:Sa} by linearity, since
$\sum_a f_a\,\omega(u)^a=f(\omega(u))$,
$\sum_a f_a\,u_0^{\,a}=f(u_0)$ and $f_0=f(0)$.
\end{proof}

\subsubsection{Prefixes ending with a horizontal step}

Let $\mathcal P$ be the set of nonempty words $p$ over $\{U,H,D\}$ such
that
\begin{itemize}
\item[(i)] $p$ ends with $H$;
\item[(ii)] every prefix of $p$ ends at a nonnegative height;
\item[(iii)] for every pair of consecutive horizontal steps of $p$,
the number of down steps lying between them is at least the height of
the first of the two.
\end{itemize}
For $p\in\mathcal P$ we denote by $h(p)$ its final height and by
$\wep(p)$ the number of its factors equal to $UD$, $UH$ or $HD$, and
we set
\[
Q(u)=1+\sum_{p\in\mathcal P}x^{|p|}\,y^{\wep(p)}\,u^{h(p)}\in\mathbb Q[y,u][[x]].
\]

\begin{lem}\label{lem:F-from-Q}
We have
\[
F_T(x,y)=\frac{Q(0)-1}{x}.
\]
\end{lem}

\begin{proof}
We show that the map $m\mapsto mH$ is a bijection between the set  $\mathcal A$ of the admissible paths 
and the set of words $p\in\mathcal P$ with $h(p)=0$, which increases
the length by one and preserves the number of weak peaks.

Let $m\in\mathcal A$. Conditions (i) and (ii) are clear for $mH$.
As for (iii), we only need to check the pair formed by the last
horizontal step of $m$, say at height $h$, and the final $H$. The steps
between them form a word over $\{U,D\}$ going from height $h$ to height
$0$, hence containing at least $h$ down steps. Moreover, the last step
of $m$ (if any) is either $D$ or $H$, so the final factor of $mH$ is
$DH$ or $HH$, which is not a weak peak. Conversely, if $p\in\mathcal P$
and $h(p)=0$, then removing the final $H$ we obtain a Motzkin path,
which is admissible by (iii).
\end{proof}

\begin{prop}\label{prop:prefix-decomposition}
Every $p\in\mathcal P$ can be uniquely written in one of the following
two forms:
\begin{itemize}
\item[(a)] $p=w\,H$, where $w$ is a word over $\{U,D\}$ nonnegative
from $0$;
\item[(b)] $p=p'\,w\,H$, where $p'\in\mathcal P$ and $w$ is a good word
from $a=h(p')$, namely a word over $\{U,D\}$ nonnegative from $a$ with
$\#D(w)\ge a$.
\end{itemize}
Conversely, every word of these forms belongs to $\mathcal P$. Moreover,
$h(p)=e_0(w)$ and $\wep(p)=\operatorname{wt}(w)$ in case (a), while
$h(p)=e_{h(p')}(w)$ and $\wep(p)=\wep(p')+\operatorname{wt}(w)$ in
case (b).
\end{prop}

\begin{proof}
Let $p\in\mathcal P$. If $p$ contains exactly one horizontal step, then
$p$ is of the form (a). Otherwise, let $p'$ be the prefix of $p$ ending
with the second-to-last horizontal step of $p$, and let $w$ be the word
over $\{U,D\}$ between this step and the last one, so that
$p=p'wH$. Clearly $p'$ satisfies (i), (ii) and (iii); moreover the last
step of $p'$ is a horizontal step at height $a=h(p')$, and by (iii)
applied to the last two horizontal steps of $p$ the word $w$ contains
at least $a$ down steps. Finally, $w$ is nonnegative from $a$ by (ii).
Uniqueness is clear, since the decomposition is determined by the
positions of the horizontal steps. The converse statement is
straightforward: conditions (ii) and (iii) for $p'wH$ follow from the
same conditions for $p'$, the nonnegativity of $w$ from $h(p')$ and
the inequality $\#D(w)\ge h(p')$.

The statement about the final height is obvious. As for the weak peaks,
in case (b) the factors of length two of $p'wH$ are those of $p'$,
those of $w$, the factor formed by the last $H$ of $p'$ and the first
step of $w$ (or the factor $HH$, if $w$ is empty), and the factor
formed by the last step of $w$ and the final $H$. Exactly as in the
proof of Lemma~\ref{lem:weight-segments}, the weak peaks among the
last three kinds of factors are counted by $\operatorname{wt}(w)$. Case
(a) is analogous, recalling that $w$ cannot start with $D$.
\end{proof}

\subsubsection{The functional equation}

\begin{thm}\label{thm:FE4231}
The series $Q(u)$ is the unique element in $\mathbb Q[y,u][[x]]$ satisfying
\begin{equation}\label{eq:FE4231}
Q(u)=1+x(1-y)\,Q(0)+\frac{xy}{K(u)}\Bigl(u\,Q\bigl(\omega(u)\bigr)-\lambda(u)\,Q(u_0)\Bigr).
\end{equation}
Moreover,
\[
F_T(x,y)=\frac{Q(0)-1}{x}
=\frac{1}{1-(1-y)x}\left(1-y+\frac{y\,Q(u_0)}{1-xu_0}\right).
\]
\end{thm}

\begin{proof}
Set $P(u)=Q(u)-1=\sum_{a\ge0}P_a\,u^a$, with $P_a\in\mathbb Q[y][[x]]$.
The words $p\in\mathcal P$ with $h(p)=a$ have length at least $a+1$,
hence the valuation of $P_a$ is at least $a+1$.

We have $$
P(u)=x\,S_0(u)+x\sum_{a\ge0}P_a\,S_a(u).
$$
In fact, Proposition~\ref{prop:prefix-decomposition} provides the two possible forms of a path in $\mathcal P.$ Recalling that the final $H$
contributes to the generating function a factor $x$ and does not change the height, the words of
the form (a) yield $x\,W_0(u)=x\,S_0(u)$, while the words of the
form (b) whose prefix $p'$ ends at height $a$ yield
$x\,P_a\,S_a(u).$ 

Applying Proposition~\ref{prop:Sa} with $a=0$ to the first summand and
Corollary~\ref{cor:operator} with $f=P$ to the second one, we obtain
\[
P(u)=x(1-y)\bigl(1+P(0)\bigr)
+\frac{xy}{K(u)}\Bigl(u\bigl(1+P(\omega(u))\bigr)-\lambda(u)\bigl(1+P(u_0)\bigr)\Bigr),
\]
which is \eqref{eq:FE4231}, since $Q=1+P$.

As for uniqueness, the same computation shows that \eqref{eq:FE4231} is
equivalent to
\[
P=x\,S_0+x\,\mathcal S[P],
\qquad\text{where}\qquad
\mathcal S[f]=\sum_{a\ge0}f_a\,S_a(u).
\]
The operator $\mathcal S$ is linear and does not decrease the
valuation, since each $S_a$ has nonnegative valuation. If $P$ and
$\tilde P$ are two solutions, then $P-\tilde P=x\,\mathcal S[P-\tilde P]$,
so the valuation of $P-\tilde P$ is strictly larger than itself, which
forces $P=\tilde P$.

The first expression for $F_T(x,y)$ is Lemma~\ref{lem:F-from-Q}. To obtain the
second one, multiply \eqref{eq:FE4231} by $K(u)$ and set $u=0$. Since
$K(0)=-x$ and $\lambda(0)=\frac{x}{1-xu_0}$, we get
\[
-x\,Q(0)=-x\bigl(1+x(1-y)Q(0)\bigr)-\frac{x^2y\,Q(u_0)}{1-xu_0},
\]
that is,
\[
\bigl(1-(1-y)x\bigr)Q(0)=1+\frac{xy\,Q(u_0)}{1-xu_0}.
\]
Hence
\[
Q(0)-1=\frac{x}{1-(1-y)x}\left(1-y+\frac{y\,Q(u_0)}{1-xu_0}\right),
\]
and the claim follows from Lemma~\ref{lem:F-from-Q}.
\end{proof}

\begin{cor}\label{cor:FE4231-y1}
The generating function $F(x)=F(x,1)$ of the involutions in
$I_n(4321,4231)$ is given by
\[
F(x)=\frac{u_0}{x}\,Q(u_0),
\qquad u_0=\frac{1-\sqrt{1-4x^2}}{2x},
\]
where $Q(u)=Q(x,1;u)$ is the unique element of $\mathbb Q[u][[x]]$
such that
\[
Q(u)=1+\frac{x}{u-x-xu^2}\left(u\,Q\Bigl(\frac{x}{1-xu}\Bigr)-u_0\,Q(u_0)\right).
\]
\end{cor}

\begin{proof}
For $y=1$ we have $K(u)=u-x-xu^2$, $\omega(u)=\frac{x}{1-xu}$ and
$\lambda(u)=\frac{x}{1-xu_0}=\omega(u_0)=u_0$ by
Lemma~\ref{lem:kernel-identities}. The claim follows from
Theorem~\ref{thm:FE4231}, since $\frac{1}{1-xu_0}=\frac{u_0}{x}$.
\end{proof}

\begin{oss}
Equation~\eqref{eq:FE4231} allows us to compute the coefficients of
$F(x,y)$ by iteration, since the operator $x\,\mathcal S$ strictly
increases the valuation. We obtain
\begin{align*}
F(x,y)={}&1+x+(1+y)x^2+(1+y)^2x^3+(1+4y+3y^2)x^4\\
&+(1+6y+9y^2+2y^3)x^5+(1+9y+19y^2+8y^3+y^4)x^6+\cdots,
\end{align*}
in accordance with the direct enumeration of $I_n(4321,4231)$.
\end{oss}

% \begin{oss}\label{oss:kernel-fails}
% The classical kernel method does not allow us to solve
% \eqref{eq:FE4231}. Indeed, by Lemma~\ref{lem:kernel-identities},
% $u_0$ is a fixed point of both $\omega$ and $\lambda$, so that
% replacing $u$ by $u_0$ in \eqref{eq:FE4231} multiplied by $K(u)$ yields
% the trivial identity $0=0$, and the unknown $Q(u_0)$ remains
% undetermined.

%------
% The map $\omega$ is a M\"obius transformation with fixed
% points $u_0$ and $1/u_0$, and multiplier
% $\omega'(u_0)=\frac{x^2y}{(1-xu_0)^2}$. In the coordinate
% $t=\frac{u-u_0}{u-1/u_0}$ it becomes $t\mapsto\omega'(u_0)\,t$, so that
% \eqref{eq:FE4231} is a $q$-difference-type equation, where the role of
% $q$ is played by a formal power series in $x$.

%\end{oss}

It is natural to ask whether $F$ is D-finite; we conjecture that it is not. This is
supported by the following fact: we computed the first $421$
coefficients of $F(x,1)$ and found no linear recurrence with polynomial
coefficients of order $r\le12$ and degree $d$, for any pair $(r,d)$
such that the number $(r+1)(d+1)$ of unknown coefficients is at most
$380$ (for instance, order $1$ and degree up to $189$, or order $12$ and
degree up to $27$). 

% As a sanity check, the same procedure immediately
% detects the recurrence satisfied by the Motzkin numbers.

\begin{question}\label{q:4231}
Is the generating function $F_T(x,1)$ of the involutions in
$I_n(4321,4231)$ non-D-finite? More generally, is the generating function $F_T(x,y)$ D-finite?
\end{question}

\subsubsection{An explicit continued fraction for $F(x,y)$}
\label{subsubsec:CF4231}

% As observed in Remark~\ref{oss:kernel-fails}, the kernel method does not
% determine the unknown series $Q(u_0)$ in~\eqref{eq:FE4231} 

In this subsection, we solve equation
\eqref{eq:FE4231}. After a change of variable, the
substitution $u\mapsto\omega(u)$ becomes a dilation, and the functional
equation turns into a three-term recurrence, which leads to a continued
fraction.

\begin{thm}\label{thm:CF4231}
Let
\[
u_0=\frac{1+(1-y)x^2-\sqrt{\bigl(1+(1-y)x^2\bigr)^2-4x^2}}{2x},
\qquad
\sigma=1-(1-y)\,x\,u_0,
\qquad
q=\frac{y\,u_0^2}{\sigma^2}.
\]
Then
\[
F(x,y)=\frac1x\left(\frac{1+\xi}{1-(1-y)\,x\,(1+\xi)}-1\right),
\qquad
\xi=\cfrac{y\,u_0\,(1-u_0^2)}{\sigma\,(1-u_0^2)-y\,u_0
-\cfrac{y^2u_0^4\,q}{\Phi}},
\]
where
\[
\Phi=A_1-\cfrac{y^2u_0^4\,q^3}{A_2-\cfrac{y^2u_0^4\,q^5}{A_3-\cfrac{y^2u_0^4\,q^7}{A_4-\cdots}}},
\qquad
A_n=(1-u_0^2)^2-y\,u_0\,(1+u_0^2)\,q^n .
\]
\end{thm}
\proof
Set $$\hat u(t)=\frac{u_0+t}{1+u_0t}.$$

The following identities, which hold because $K(u_0)=0$, are readily checked with a computer algebra system,
\begin{gather*}
    q=\frac{x^2y}{(1-xu_0)^2},\quad K\bigl(\hat u(t)\bigr)=\frac{x\,t\,(1-u_0^2)^2}{u_0(1+u_0t)^2},\\ (1+u_0t)\,\lambda\bigl(\hat u(t)\bigr)=u_0+\zeta\,t\quad \text{where }\zeta=\frac{x\bigl(u_0-(1-y)x\bigr)}{1-xu_0},\text{ and }\\
    \omega\bigl(\hat u(t)\bigr)=\hat u(qt).
\end{gather*}

Notice that the last of these identities shows that, after the change of variable $u\to \hat u,$ the map $\omega$ acts as a dilation. 

Since the coefficient of $x^m$ in $Q(u)$ is a polynomial in $u$, we can
substitute $u=\hat u(t)$ and obtain a formal power series
\[
\tilde Q(t)=Q\bigl(\hat u(t)\bigr)=\sum_{n\ge0}d_n\,t^n,
\qquad d_n\in\mathbb Q[y][[x]],
\]
with $d_0=Q(u_0)$. By the last of the previous identities, the same substitution
transforms $Q(\omega(u))$ into $\tilde Q(qt)$.

We multiply \eqref{eq:FE4231} by
$K(u)$, substitute $u=\hat u(t)$, and multiply by
$u_0(1+u_0t)^2/x$. We obtain
\[
t\,(1-u_0^2)^2\,\tilde Q(t)=t\,(1-u_0^2)^2\,E
+y\,u_0\,(1+u_0t)\Bigl((u_0+t)\,\tilde Q(qt)-(u_0+\zeta t)\,d_0\Bigr),
\]
where $E=1+x(1-y)Q(0).$

% The coefficient of $t^m$ in the large parenthesis is $0$ for $m=0$,
% $u_0q\,d_1+(1-\zeta)\,d_0$ for $m=1$, and $u_0q^md_m+q^{m-1}d_{m-1}$ for
% $m\ge2$. Extracting the coefficients of $t^1$, $t^2$ and $t^{n+1}$
% ($n\ge2$) we get, respectively, \eqref{eq:rec0}, \eqref{eq:rec} for
% $n=1$, and \eqref{eq:rec} for $n\ge2$.
Extracting the coefficients of $t^n,$ $n\geq 1,$ in the previous identity we get
\begin{align}
y\,u_0^2\,q\,d_1&=\bigl((1-u_0^2)^2-y\,u_0(1-\zeta)\bigr)\,d_0-(1-u_0^2)^2\,E,
\label{eq:rec0}\\  y\,u_0^2\,q^{2}\,d_{2}&=A_1\,d_1-y\,u_0^2\,(1-\zeta)\,d_{0},\label{eq:rec1}\\
y\,u_0^2\,q^{n+1}\,d_{n+1}&=A_n\,d_n-y\,u_0^2\,q^{n-1}\,d_{n-1}
\qquad(n\ge2).
\label{eq:rec}
\end{align}

The passage from a three-term linear recurrence to a continued fraction
is classical: the ratios of consecutive terms of a suitable solution of
such a recurrence are given by a continued fraction built from its
coefficients (see, e.g.,~\cite[Section 5.3]{Jones_Thron_1984}).
We now apply this ideas to the recurrence~\eqref{eq:rec}

Define the continued fraction
\begin{equation}\label{eq:CF-r}
r=\cfrac{y\,u_0^2\,(1-\zeta)}{A_1-\cfrac{y^2u_0^4\,q^3}{A_2-\cfrac{y^2u_0^4\,q^5}{A_3-\cdots}}}\ .
\end{equation}
It is a well-defined formal power series, since each $A_n$ has constant
term $1$ and the $n$-th partial numerator $y^2u_0^4q^{2n+1}$ is divisible
by $x^{4n+6}$. For $n\ge1$ let $r_n$ be the continued fraction obtained
from \eqref{eq:CF-r} by starting at level $n$, so that $r=r_1$ and
\begin{equation}\label{eq:rn}
r_n=\frac{y\,u_0^2\,q^{n-1}}{A_n-y\,u_0^2\,q^{n+1}\,r_{n+1}}\qquad(n\ge 2).
\end{equation}

Our aim now is to show that $d_1=r\,d_0$. For $n\ge1$ set $e_n=d_n-r_nd_{n-1}$
and $w_n=A_n-y\,u_0^2q^{n+1}r_{n+1}$, which has constant term $1$. By
\eqref{eq:rn} we have $w_nr_n=y\,u_0^2q^{n-1}$, hence, by
\eqref{eq:rec},
\begin{align*}
y\,u_0^2q^{n+1}e_{n+1}
&=A_nd_n-y\,u_0^2q^{n-1}d_{n-1}-y\,u_0^2q^{n+1}r_{n+1}d_n\\
&=w_nd_n-w_nr_nd_{n-1}=w_n\,e_n .
\end{align*}
Iterating, for every $N\ge1$ we get
\[
e_1=\prod_{n=1}^{N}\frac{y\,u_0^2\,q^{n+1}}{w_n}\cdot e_{N+1}.
\]
Since $e_{N+1}$ is a formal power series and $y\,u_0^2q^{n+1}$ is
divisible by $x^{2n+4}$, the series $e_1$ is divisible by $x^{N^2+5N}$
for every $N$; hence $e_1=0$, that is, $d_1=r\,d_0$, as claimed above. 

Substituting $d_1=r\,d_0$ into \eqref{eq:rec0} we get
$$Q(u_0)=d_0=\frac{(1-u_0^2)^2}{(1-u_0^2)^2-y\,u_0(1-\zeta)-y\,u_0^2\,q\,r}\,\bigl(1+x(1-y)Q(0)\bigr).$$ On the other
hand, in the proof of Theorem~\ref{thm:FE4231} we showed that
$$\bigl(1-(1-y)x\bigr)Q(0)=1+\frac{xy\,Q(u_0)}{1-xu_0}.$$ Combining the
two relations,
\[
\bigl(1-(1-y)x\bigr)Q(0)=1+\xi\bigl(1+x(1-y)Q(0)\bigr),
\]
that is,
\[
Q(0)=\frac{1+\xi}{1-(1-y)x(1+\xi)},
\]
and the claim follows from $F=(Q(0)-1)/x$ (Lemma~\ref{lem:F-from-Q}).
\endproof

For $y=1$ the formula becomes particularly simple.

\begin{cor}\label{cor:CF4231-y1}
Let $s=\frac{1-\sqrt{1-4x^2}}{2x}$, $q=s^2$ and
$B_n=(1-q)^2-s\,q^n(1+q)$. The generating function of the involutions
in $I_n(4321,4231)$, counted by length, is
\[
F(x,1)=\cfrac{1-q^2}{1-q-s-\cfrac{s^2q^2}{B_1-\cfrac{s^2q^4}{B_2-\cfrac{s^2q^6}{B_3-\cdots}}}}\ .
\]
\end{cor}

% \begin{proof}
% For $y=1$ we have $u_0=s$ and $\frac{x}{1-xs}=s$, hence $q=s^2$,
% $\zeta=q$, $(1-u_0^2)^2=(1-q)^2$, $A_n=B_n$, $\xi=s\gamma$, and
% $F(x,1)=\xi/x=(1+q)\gamma$, since $\frac sx=1+s^2$. Moreover
% $y\,u_0^2q\,r=q^2r$ and, by~\eqref{eq:CF-r}, $r=\frac{q(1-q)}{\Psi}$,
% where $\Psi=B_1-\cfrac{s^2q^4}{B_2-\cfrac{s^2q^6}{B_3-\cdots}}$. Hence
% \[
% F(x,1)=\frac{(1+q)(1-q)^2}{(1-q)^2-s(1-q)-\dfrac{s^2q^2(1-q)}{\Psi}}
% =\frac{1-q^2}{1-q-s-\dfrac{s^2q^2}{\Psi}}.
% \qedhere
% \]
% \end{proof}

\begin{table}[ht]
\centering
\renewcommand{\arraystretch}{1.3}
\begin{tabular}{>{\raggedright\arraybackslash}p{3.5cm} >{\raggedright\arraybackslash}p{1.5cm}
                >{\raggedright\arraybackslash}p{6.5cm}
                >{\raggedright\arraybackslash}p{3.5cm} }
\toprule
$T$ & Section & $\mathcal W^{20}_4{(T)}$ & Generating function \\
\midrule
$\emptyset$ & \ref{4321}             & None                                  & algebraic \\
$\{1234\}$   & \ref{1234}            & $DDU ,\
   DUU,\
   DDDD,\
   DDDH$
   $DDHH,\
   DHHH,\
   DHHU,\
   HHHH$
   $HHHU,\
   HHUU,\
   HUUU,\
   UUUU  $                           & finite \\
$\{1243\}$   & \ref{1243}            & $HHU,\ DDU,\ DHU,\ DDH$             & rational \\ 
$\{1342,1423\}$  & \ref{1342}        & $DU,\ DHH,\ HUU,\ HUHH,\ DHDD$       & rational \\
$\{1432\}$  & \ref{1432}             & $HUH,\ DUH$                          & algebraic \\
$\{2143\}$  & \ref{2143}             & $DU,\ HUDH$                             & rational \\

$\{2413,3142\}$ & \ref{2413}         & $DU,\ DH,\ HU$                       & rational \\
$\{3241,4213\}$  & \ref{3241}        & $HU,\ HDH$                           & algebraic \\
$\{3412\}$   & \ref{3412}            & $UU,\ DD,\ DU,\ DH,\ HU$             & rational \\
$\{1324\}$   & \ref{1324}            & $^\ast$ $DHU,\ 
   DHDH,\
   DUDU,\
   HUHU$                    & rational \\

$\{2341,4123\}$  & \ref{2341}        & $^\ast$ $DU,\
   DDD,\
   DDH,\
   DHH,\ $
   $HHU,\
   HUU,\
   UUU,\
   HHHH$                         & rational \\

$\{3421,4312\}$ & \ref{3421}         & None                    & algebraic \\

$\{4231\}$    & \ref{4231}           & None                             & not D-finite? \\

\bottomrule
\end{tabular}
\caption{Involutions avoiding $4321$ and the patterns in $T$: minimal
forbidden subsequences (in the sense of
Definition~\ref{def:subsequence}) of the Motzkin paths associated with
connected involutions, as suggested by the procedure of
Section~\ref{sec:methodology}, and nature of the bivariate generating
function. Sets $T$ equivalent under the symmetries of Section~2 are
omitted.\\
$^{\ast}$In these cases, the forbidden subsequences do not
characterize the connected Motzkin paths, and additional conditions are needed.}
\label{tab:summary}
\end{table}

\section{Conclusion and open problems}
\label{sec:conclusion}

We have enumerated, according to length and number of descents, all
classes of involutions avoiding $4321$ and another pattern of length
four, combining Biane's bijection with a computer search for the
subsequences avoided by the associated Motzkin paths. All generating
functions turn out to be rational or algebraic, with the exception of
$I_n(4321,4231)$, whose generating function is an explicit continued
fraction and is conjecturally not D-finite (Question~\ref{q:4231}).
Natural open problems are: proving this conjecture; finding a
combinatorial proof of the continued fraction, in the spirit
of~\cite{flaj}; and extending the method to involutions
avoiding $4321$ together with longer patterns.

\FloatBarrier
% \section*{Acknowledgments}

% We would like to express our sincere gratitude to the anonymous referees, whose insightful comments and careful reading greatly enhanced the quality of this paper. 

\addcontentsline{toc}{section}{Bibliography}
\bibliographystyle{plain}
\bibliography{BIBLIOGRAFIA}

\end{document}